\documentclass[11pt]{article}
\usepackage{amsmath,amssymb,amsthm,amscd,esint}

\numberwithin{equation}{section}

\DeclareMathOperator{\sing}{sing}

\DeclareMathOperator{\codim}{Codim}

\DeclareMathOperator{\supp}{Supp}

\DeclareMathOperator{\vol}{vol}
\DeclareMathOperator{\tr}{tr}

\DeclareMathOperator{\diam}{diam}

\DeclareMathOperator{\re}{Re}

\DeclareMathOperator{\Ric}{Ric}
\DeclareMathOperator{\id}{id}
\DeclareMathOperator{\iso}{iso}
\DeclareMathOperator{\reg}{reg}
\DeclareMathOperator{\amp}{amp}

\DeclareMathOperator{\exc}{exc}

\def\CC{{\mathbb C}}

\newtheorem{prop}{Proposition}[section]
\newtheorem{theo}[prop]{Theorem}
\newtheorem{lemm}[prop]{Lemma}
\newtheorem{clai}[prop]{Claim}
\newtheorem{coro}[prop]{Corollary}
\newtheorem{rema}[prop]{Remark}

\def\begeq{\begin{equation}}
\def\endeq{\end{equation}}

\begin{document}

\title{Bounding diameter of singular K\"ahler metric}

\author{Gabriele La Nave\thanks{Email: lanave@illinois.edu}\\
University of Illinoise at Urbana-Champaign
\\[5pt]
Gang Tian\thanks{Supported partially by
NSF and NSFC grants. Email: tian@math.princeton.edu}\\
Beijing University and Princeton University
\\[5pt]
Zhenlei Zhang\thanks{Supported partially by NSFC 11371256 and Chinese Scholarship Council. Email: zhleigo@aliyun.com}\\
Capital Normal University}
\date{}

\maketitle

\begin{abstract}
In this paper we investigate the differential geometric and algebro-geometric properties of the noncollapsing limit in the continuity method that was introduced by the first two named authors in \cite{LaTi14}.
\end{abstract}

\tableofcontents



\section{Introduction}

The Analytic Minimal Model Program through K\"ahler-Ricci flow was initiated by the second named author and his collaborators
(\cite{TiZh06}, \cite{SoTi07, SoTi12, SoTi09}, \cite{Ti08}). Many progresses have been made on the program (see
\cite{FoZh12}, 
\cite{So09, So13, So14, So14-2},
\cite{SoWe11, SoWe13, SoWe14},
\cite{ToWeYa}, \cite{Zh10} et al)
in addition to what mentioned above. However, there are some
serious difficulties in studying the singularity formation of the K\"ahler-Ricci flow because we do not know how to bound the Ricci curvature from below along the flow. To overcome these difficulties, in \cite{LaTi14}, the first and second named authors introduced a new continuity method. It provides an alternative way of studying the Analytic Minimal Model Program and
has the advantage of having Ricci curvature bounded from below along the deformation, so many analytic tools become available.

In this paper we investigate the differential geometric and algebro-geometric properties of the limit in the continuity method. We will focus on the noncollapsing case and confirm some conjectures proposed in \cite{LaTi14} under the algebraic assumption of the K\"ahler class. This may be seen as the first step toward the Minimal Model Program through the continuity equation.

The method to study the algebraic structure of the limit space is provided by the partial $C^0$ estimate for a polarized line bundle on a K\"ahler manifold, following the recent work of Donaldson-Sun \cite{DoSu14} and Tian \cite{Ti13} (also see \cite{Ti12}).
In our case, however, the considered will be a ``limit line bundle'' of polarized bundles, so it is not canonically related to the K\"ahler metric on the underlying manifold. In general, this may cause some problem in finding numerous of holomorphic peak sections. Similar problem also appears in K\"ahler-Ricci flow \cite{So14}. In \cite{So14}, Song proved the convergence of the normalized K\"ahler-Ricci flow on a minimal model of general type in the Cheeger-Gromov sense to a compact singular K\"ahler-Einstein metric on its canonical model. Song developed some analysis on the limit space to prove the theorem, namely the gradient estimate to the limit potential function, the gradient estimate to the holomorphic sections of canonical line bundle, the H\"ormander $L^2$ estimate and a diameter estimate to the singular K\"ahler-Einstein metric. These depend on the property of (real) codimension 4 singularity in the Cheeger-Gromov limit space so that one can construct good cut-offs on both the limit space and its tangent cones. In our case, the H\"ormander $L^2$ estimate holds automatically for the continuity method. The main difficulty in the continuity method is that the Cheeger-Gromov limit as well as its tangent cones may have (real) codimension 2 singularities. However, one can show by recursion that if this happens at a point in the tangent cone, then the singular set around that point must be a locally analytical set modeled by taking the limit of a divisor (ample locus of the limit K\"ahler class) on the original manifold. The picture fits into the case of convergence of conical K\"ahler-Einstein manifolds as considered in \cite{Ti12}.

We begin with a compact projective manifold $M$ with a K\"ahler metric $\omega_0\in qc_1(L^{'})$ where $L'$ is a line bundle, $q$ is a positive rational number. Consider a 1-parameter family of equations \cite{LaTi14}
\begin{equation}\label{M-A: continuity}
\omega=\omega_0-t\Ric(\omega).
\end{equation}
The equation is solvable up to
\begin{equation}
T=:\sup\{t|[\omega_0]-tc_1(M)>0\}.
\end{equation}

In the special case when the limit does not collapse, we have the following theorem.

\begin{theo}[\cite{LaTi14}]\label{convergence: current}
Assume $T<\infty$ and $([\omega_0]-Tc_1(M))^n>0$, then $\omega_t$ converge to a unique weakly K\"ahler metric $\omega_T$ which is smooth outside of a subvariety $\mathcal{S}_M$ and satisfies
\begin{equation}
\omega_T=\omega_0-T\Ric(\omega_T),\,\mbox{ on }M\backslash\mathcal{S}_M.
\end{equation}
\end{theo}

Here, the subvariety $\mathcal{S}_M$ equals the non-ample locus of the cohomological class
\begin{equation}
\mathcal{S}_M=\bigcap\{D| D\mbox{ is a divisor satisfying }[\omega_0]-Tc_1(M)-\epsilon[D]>0 \mbox{ for some }\epsilon>0\}.
\end{equation}

It is conjectured in \cite[Conjecture 4.1]{LaTi14} that the limit space has more regular properties, such as metric structure and algebraic structure, and the convergence of $(M,\omega_t)$ takes place in the Cheeger-Gromov topology. In this note we confirm this conjecture partially.

\begin{theo}\label{convergence: CG}
Assume as in above theorem. Then
\begin{itemize}

\item[{\rm(1)}] $(M,\omega_t)$ converges in the Cheeger-Gromov topology to a compact path metric space $(M_T,d_T)$ which is the metric completion of $(M\backslash\mathcal{S}_M,\omega_T)$;

\item[{\rm(2)}] $M_T$ has regular/singular decomposition $M_T=\mathcal{R}\cup\mathcal{S}$, a point $x\in\mathcal{R}$ iff the tangent cone at $x$ is $\mathbb{C}^n$;

\item[{\rm(3)}] $\mathcal{S}$ is closed and has real codimension $\ge 2$, $\mathcal{R}$ is geodesically convex;

\item[{\rm(4)}] $M_T$ is homeomorphic to a normal projective variety with $\mathcal{S}$ corresponding to a subvariety.



\end{itemize}
\end{theo}

\begin{rema}
Our proof here used the Kawamata base point free theorem. We hope to remove this constraint eventually and
prove the Kawamata base point free theorem for a general line bundle by
applying a refinement of our method here together with some arguments in {\rm\cite{So14-2}}.
\end{rema}

\begin{rema}
The next step toward the Minimal Model Program is to construct flips and deform the continuity equation {\rm(\ref{M-A: continuity})} through the singular time $T$. This is basically the same as the K\"ahler-Ricci flow in {\rm\cite{SoTi09}}. We will leave the discussion in a forthcoming paper.
\end{rema}

\noindent
{\bf Acknowledgement}: The third named author is visiting the Math Department of Princeton University. He would like to thank the department for the nice working environment. He also want to thank Jian Song for his discussion on the Proposition 3.22.



\section{A prior estimate to the continuity equation}

In this section, we present some estimate to the continuity equation (\ref{M-A: continuity}). By a dilation of the cohomological class $[\omega_0]$ and the time parameter, we may assume that $[\omega_0]=c_1(L^{'})$ for some line bundle $L^{'}$. The rationality theorem of Kawamata \cite{Ka84} says that $T\in\mathbb{Q}$. Take a positive integer $\ell_0$ such that $T\ell_0\in\mathbb{Z}$, and define the``limit line bundle'' $L=\ell_0(L^{'}+TK_M)$.

Let $\omega_t$, $t\in[0,T)$, be the solution to (\ref{M-A: continuity}) with initial metric $\omega_0$. Since the limit class $[\omega_0]+TK_M$ is nef and big, according to the base point free theorem \cite{KoMo}, we may assume $\ell_0$ is chosen such that $L$ has no base points. Let $H_0$ be any Hermitian metric on $L$. An orthonormal basis of $(L,H_0)$ at time $t=0$ gives a holomorphic map
\begin{equation}
\Phi:M\rightarrow\mathbb{C}P^N
\end{equation}
where $N=\dim H^0(M;L)-1$. Denote by $\omega_{FS}$ the Fubini-Study metric of $\mathbb{C}P^N$, then the pull back form $\eta_T=\frac{1}{\ell_0}\Phi^*\omega_{FS}$ defines a smooth metric on $M_{\reg}$, the set of regular points of $\Phi$.

 By putting $\eta_t=\frac{T-t}{T}\omega_0+\frac{t}{T}\eta_T,$ a family of background metrics, the solution $\omega_t$ can be written as
\begin{equation}
\omega_t=\eta_t+\sqrt{-1}\partial\bar{\partial}u_t.
\end{equation}
Since $\frac{1}{T}(\omega_0-\eta_T)\in c_1(M)$, there is a smooth volume form $\Omega$ on $M$ such that
\begin{equation}
\Ric(\Omega)=\frac{1}{T}(\omega_0-\eta_T).
\end{equation}
Then, the continuity equation (\ref{M-A: continuity}) is equivalent to
\begin{equation}\label{M-A: continuity 2}
\big(\eta_t+\sqrt{-1}\partial\bar{\partial}u_t\big)^n=e^{\frac{u_t}{t}}\Omega.
\end{equation}
By \cite{LaTi14}, the solution $(M,\omega_t)$ converges smoothly outside $\mathcal{S}_M$ to a limit metric $\omega_T$ satisfying
\begin{equation}
\Ric(\omega_T)=\frac{1}{T}(\omega_0-\omega_T),\,\mbox{ on }M\backslash\mathcal{S}_M.
\end{equation}

\begin{lemm}
There is a constant $C$ independent of $t$ such that
\begin{equation}\label{potential bound: C^0}
\|u_t\|_{C^0}\le C,\,\forall t\in[\frac{T}{2},T).
\end{equation}
\end{lemm}
\begin{proof}
The uniform upper bound of $u_t$ is trivial consequence of the maximum principle. The $L^\infty$ bound follows from the capacity calculation of Ko{\l}odziej \cite{Ko98}; see also \cite{Zh06} or \cite{EyGuZe09} for exactly our case when $u_t$ has a uniform upper bound.
\end{proof}

\begin{coro}
There exists $C$ independent of $t$ such that
\begin{equation}
C^{-1}\Omega\le \omega_t^n\le C\Omega,\,\forall t\in[\frac{T}{2},T).
\end{equation}
\end{coro}

\begin{lemm}
There exists $C$ independent of $t$ such that
\begin{equation}
u^{'},\,u^{''}\le C,\,\forall \frac{T}{2}\le t<T.
\end{equation}
\end{lemm}
\begin{proof}
If there is no confusion, we simply denote $\omega_t$ by $\omega$. Differentiating the Monge-Amp\`ere equation (\ref{M-A: continuity 2}) to get
\begin{equation}\label{M-A: continuity 3}
\tr_\omega\omega^{'}=\frac{1}{t^2}\big(tu^{'}-u\big)
\end{equation}
where
\begin{equation}\label{M-A: continuity 3.5}
\omega^{'}=\frac{1}{T}(\eta_T-\omega_0)+\sqrt{-1}\partial\bar{\partial}u^{'}
=\frac{1}{t}\big(\omega-\omega_0-\sqrt{-1}\partial\bar{\partial}u\big)
+\sqrt{-1}\partial\bar{\partial}u^{'}.
\end{equation}
Thus,
\begin{equation}\label{M-A: continuity 4}
\triangle(tu^{'}-u)=\tr_\omega\big(t\omega^{'}-\omega+\omega_0\big)
=\frac{1}{t}\big(tu^{'}-u\big)+\tr_\omega\omega_0-n.
\end{equation}
Then applying the maximum principle we derive $tu^{'}-u\le C$. To get the upper bound of $u^{''}$ we first observe that (\ref{M-A: continuity 3}) gives $tu^{'}-u=t^2\tr_\omega\omega^{'}$. Taking derivation,
$$tu^{''}=2t\tr_\omega\omega^{'}+t^2\triangle u^{''}-t^2|\omega^{'}|^2=t^2\triangle u^{''}-|\omega-t\omega^{'}|^2+n.$$
So,
\begin{equation}\label{M-A: continuity 5}
t^2\triangle u^{''}=tu^{''}+|\omega-t\omega^{'}|^2-n.
\end{equation}
Then apply the maximum principle.
\end{proof}

\begin{lemm}
The function $u_t$ converges uniformly to a bounded function $u_T$ satisfying
\begin{equation}\label{M-A: continuity 6}
\big(\eta_T+\sqrt{-1}\partial\bar{\partial}u_T\big)^n=e^{\frac{u_T}{T}}\Omega
\end{equation}
in the current sense.
\end{lemm}
\begin{proof}
By (\ref{M-A: continuity 4}),
$$t\triangle\big(\frac{u}{t}\big)^{'}=\big(\frac{u}{t}\big)^{'}+\frac{1}{t}\big(\tr_\omega\omega_0-n\big)
\ge\big(\frac{u}{t}\big)^{'}-\frac{n}{t}.$$
Thus,
$$t\triangle\big(\frac{u}{t}-n\log t\big)^{'}\ge\big(\frac{u}{t}-n\log t\big)^{'}.$$
So, $\big(\frac{u}{t}-n\log t\big)$ is monotone decreasing. Consequently, $u_t$ converges uniformly to a unique limit $u_T$. It is obvious that $u_T$ is smooth outside $\mathcal{S}_M$ and (\ref{M-A: continuity 6}) holds in the current sense.
\end{proof}

\begin{rema}
Similarly, the formula {\rm(\ref{M-A: continuity 5})} imply that $\big(u^{'}-n\log t\big)^{'}\le 0.$
In particular, $\big(u^{'}-n\log t\big)$ decrease to unique functions on $M$ which is locally bounded on $M\backslash\mathcal{S}_M$.
\end{rema}

\begin{rema}
The function $u_T$ is the unique solution to the Monge-Amp\`ere equation {\rm(\ref{M-A: continuity 6})} in the big class $[\omega_0]+TK_M$, cf. {\rm\cite{BEGZ}}.
\end{rema}

\begin{prop}\label{metric comparison: prop}
There exists $C$ independent of $t$ such that
\begin{equation}\label{metric comparison: T}
\eta_T\le C\omega_t,\,\forall t\in[\frac{T}{2},T).
\end{equation}
\end{prop}
\begin{proof}
By Yau's Schwarz lemma \cite{Ya78-2}, using $\Ric(\omega_t)\ge-\frac{1}{t}\omega_t$,
$$\triangle_\omega\log\tr_\omega\eta_T\ge-\frac{n}{t}-n\tr_\omega\eta_T.$$
On the other hand, $\eta_t\ge\epsilon\eta_T$ for some $\epsilon>0$ independent of $t$, so
$$\triangle u=n-\tr_\omega\eta_t\le n-\epsilon\tr_\omega\eta_T.$$
Hence,
$$\triangle_\omega(\log\tr_\omega\eta_T-\frac{2n}{\epsilon}u)\ge-\frac{C(n,T)}{\epsilon}+n\tr_\omega\eta_T,\,\forall t\in[\frac{T}{2},T).$$
Notice that the maximal of $\big(\log\tr_\omega\eta_T-\frac{2n}{\epsilon}u\big)$ is achieved at some point $x_0$ where $\eta_T\neq 0$. At this maximum point, $\tr_\omega\eta_T(x_0)\le C$. The maximum principle gives the desired estimate $\tr_\omega\eta_T\le C$ for any $\frac{T}{2}\le t<T$.
\end{proof}

\begin{coro}\label{smooth on regular locus}
The limit metric $\omega_T$ is smooth on the regular set $M_{\reg}$.
\end{coro}
\begin{proof}
On any compact subset $K\subset M_{\reg}$ the measure $\eta_T^n$ is uniformly equivalent to $\Omega$. So, by (\ref{M-A: continuity 6}) and (\ref{metric comparison: T})
$$C^{-1}\eta_T\le \omega_T\le C_K\eta_T,\mbox{ on }K,$$
for some constant $C_K$. In particular $n+\triangle_{\eta_T}u_T\le C_K$ on $K$. Then applying a bootstrap argument we get the higher derivative bound $\|u_T\|_{C^k(K)}\le C_{k,K}$ on $K$.
\end{proof}

\begin{rema}
In Section 3 we will show that $\omega_t$ converges smoothly to $\omega_T$ on $M_{\reg}$ and that  $M\backslash\mathcal{S}_M\subset  M_{\reg}$.
\end{rema}

Inspired by K\"ahler-Ricci flow we define $w=(T-t)u^{'}+u$ which satisfies
\begin{equation}\label{M-A: continuity 7}
\triangle w=\frac{1}{t}w+n-\tr_\omega\eta_T-\frac{T}{t^2}u.
\end{equation}
This can be seen by combining
\begin{equation}
\triangle u=n-\frac{T-t}{T}\tr_\omega\omega_0-\frac{t}{T}\tr_\omega\eta_T
\end{equation}
and
\begin{equation}\label{M-A: continuity 8}
\triangle u^{'}=\frac{1}{t^2}(tu^{'}-u)+\frac{1}{T}\tr_\omega(\omega_0-\eta_T).
\end{equation}
A direct corollary of (\ref{M-A: continuity 7}) is
\begin{equation}
\|w\|_{C^0}+\|\triangle w\|_{C^0}\le C,\,\forall t\in[\frac{T}{2},T).
\end{equation}
Combining with the $C^0$ bound of $u$ we also have
\begin{equation}
-\frac{C}{T-t}\le u^{'}\le C,\,\forall t\in[\frac{T}{2},T).
\end{equation}

\begin{prop}
There exists $C$ independent of $t$ such that
\begin{equation}\label{potential bound: C^1-1}
\|\nabla w\|_{C^0}\le C,\,\forall t\in[\frac{T}{2},T).
\end{equation}
In particular, since $u^{'}$ converges to a locally bounded function on $M\backslash\mathcal{S}_M$,
\begin{equation}\label{potential bound: C^1-2}
\|\nabla u_T\|_{C^0}\le C,\mbox{ on }M\backslash\mathcal{S}_M.
\end{equation}
\end{prop}
\begin{proof}
The authors learned the second estimate from J. Song. It is inspired by the calculation in K\"ahler-Ricci flow; see \cite{Zh10} and \cite{So14}. Recall that the continuity equation (\ref{M-A: continuity}) gives
$Ric(\omega)=\frac{1}{t}(\omega_0-\omega)$, so by the Bochner formula,
\begin{eqnarray}
\triangle|\nabla w|^2=|\nabla\nabla w|^2+|\nabla\bar{\nabla}w|^2+\nabla_i\triangle w\nabla_{\bar{i}}w+\nabla_{\bar{i}}\triangle w\nabla_iw+\frac{1}{t}\langle\omega_0-\omega,\partial w\otimes\bar{\partial}w\rangle\nonumber
\end{eqnarray}
where, by (\ref{M-A: continuity 7}),
$$\nabla_i\triangle w\nabla_{\bar{i}}w+\nabla_{\bar{i}}\triangle w\nabla_iw=\frac{2}{t}|\nabla w|^2-2\re\bigg(\nabla_i\tr_\omega\eta_T\nabla_{\bar{i}}w\bigg)-\frac{2T}{t^2}\re\bigg(\nabla_iu\nabla_{\bar{i}}w\bigg).$$
So,
\begin{eqnarray}
\triangle|\nabla w|^2&\ge&|\nabla\nabla w|^2+|\nabla\bar{\nabla}w|^2+\frac{1}{t}|\nabla w|^2-2\re\bigg(\nabla_i\tr_\omega\eta_T\nabla_{\bar{i}}w\bigg)-\frac{2T}{t^2}\re\bigg(\nabla_iu\nabla_{\bar{i}}w\bigg)\nonumber\\
&\ge&\frac{1}{2t}|\nabla w|^2-4t|\nabla\tr_\omega\eta_T|^2-\frac{16T^2}{t^3}|\nabla u|^2.\nonumber
\end{eqnarray}
To estimate the term $|\nabla\tr\eta_T|^2$ recall that by the Schwarz lemma \cite{Ya78-2}
\begin{equation}\nonumber
\triangle\tr_\omega\eta_T\ge\tr_\omega\eta_T\bigg(-\frac{n}{t}-A\tr_\omega\eta_T\bigg)
+\frac{1}{\tr_\omega\eta_T}|\nabla\tr_\omega\eta_T|^2\ge-C_1+C_2^{-1}|\nabla\tr_\omega\eta_T|^2
\end{equation}
where we used the $C^0$ bound of $\tr_\omega\eta_T$; to estimate the term $|\nabla_iu|^2$ we use
$$\triangle (-u)=-n+\frac{T-t}{T}\tr_\omega\omega_0+\frac{t}{T}\tr_\omega\eta_T\ge\frac{T-t}{T}\tr_\omega\omega_0-C_3$$
and
$$\triangle u^2=2u\triangle u+2|\nabla u|^2\ge2|\nabla u|^2-C_4\frac{T-t}{T}\tr_\omega\omega_0-C_5.$$
The constants $C_i$ here do not depend on $t\in[\frac{T}{2},T)$. A combination gives
$$\triangle\bigg(|\nabla w|^2+4tC_2\tr_\omega\eta_T+\frac{8T^2}{t^3}u^2-\frac{8T^2}{t^3}C_4u\bigg)
\ge\frac{1}{2t}|\nabla w|^2-C_6.$$
By maximum principle we get a uniform upper bound of $|\nabla w|$ when $t\in[\frac{T}{2},T)$.
\end{proof}

\begin{rema}
In \S 3.2, we will show that $M\backslash\mathcal{S}_M$ is dense in the Gromov-Hausdorff limit, so $u_T$ is globally Lipschitz.
\end{rema}



\section{Algebraic structure of the limit space}




\subsection{Preliminaries}

We start with the Bochner and Weitzenb\"och formulas on a general line bundle. Let $(M,\omega)$ be a K\"ahler manifold of dimension $n$ and $(L,h)$ be a Hermitian line bundle over $M$. Let $\Theta_h$ be the Chern curvature form of $h$. Let $\nabla$ and $\bar{\nabla}$ denote the $(1,0)$ and $(0,1)$ part of a connection respectively. The connection appeared in this paper is usually known as the Chern connection or Levi-Civita connection.

For a holomorphic section $\varsigma\in H^0(M,L)$ we write for simplicity
$$|\varsigma|=|\varsigma|_{h},\,|\nabla\varsigma|=|\nabla\varsigma|_{h\otimes\omega},$$
and
$$|\nabla\nabla\varsigma|^2=\sum_{i,j}|\nabla_i\nabla_j\varsigma|^2,
\,|\bar{\nabla}\nabla\varsigma|^2=\sum_{i,j}|\nabla_{\bar{i}}\nabla_j\varsigma|^2.$$
By direct computation we have

\begin{lemm}[Bochner formulas]
For any $\varsigma\in H^0(M,L)$ we have
\begin{equation}\label{Bochner formula: 01}
\triangle_{\omega}|\varsigma|^2=|\nabla\varsigma|^2-|\varsigma|^2\cdot\tr_{\omega}\Theta,
\end{equation}
and
\begin{eqnarray}\label{Bochner formula: 02}
\triangle_{\omega}|\nabla\varsigma|^2&=&|\bar{\nabla}\nabla\varsigma|^2+|\nabla\nabla\varsigma|^2
-\nabla_j\Theta_{i\bar{j}}\langle\varsigma,\nabla_{\bar{i}}\bar{\varsigma}\rangle
-\nabla_{\bar{j}}(\tr_{\omega}\Theta)\langle\nabla_j\varsigma,\bar{\varsigma}\rangle\nonumber\\
&&\hspace{1cm}+R_{i\bar{j}}\langle\nabla_j\varsigma,\nabla_{\bar{i}}\bar{\varsigma}\rangle
-2\Theta_{i\bar{j}}\langle\nabla_j\varsigma,\nabla_{\bar{i}}\bar{\varsigma}\rangle
-|\nabla\varsigma|^2\cdot\tr_{\omega}\Theta,
\end{eqnarray}
where $R_{i\bar{j}}$ is the Ricci curvature of $\omega$, $\langle,\rangle$ is the inner product defined by $h$.
\end{lemm}

\begin{lemm}[Weitzenb\"och formulas]
For any smooth section $\xi\in\Gamma(T^{0,1}M\otimes L)$ we have
\begin{equation}\label{Weitzenboch formula: 01}
(\bar{\partial}^*\bar{\partial}+\bar{\partial}\bar{\partial}^*)\xi=\bar{\nabla}^*\bar{\nabla}\xi+\big(\Theta
+\Ric(\omega)\big)(\xi,\cdot),
\end{equation}
and
\begin{equation}\label{Weitzenboch formula: 02}
(\bar{\partial}^*\bar{\partial}+\bar{\partial}\bar{\partial}^*)\xi=\nabla^*\nabla\xi+\Theta(\xi,\cdot)
-\big(\tr_{\omega}\Theta\big)\xi
\end{equation}
where $\Ric(\xi,\cdot)$ (similar to $\Theta(\xi,\cdot)$) is defined by, for $\xi=\alpha_{\bar{i}}d\bar{z}^i\otimes\varsigma$ in local normal coordinate,
$$\Ric(\xi,\cdot)=R_{i\bar{j}}\alpha_{\bar{i}}d\bar{z}^j\otimes\varsigma.$$
\end{lemm}

We also need a special version of the effective finite generation property of a line bundle. Suppose the line bundle $L$ satisfies in addition that (i) $L$ is base point free and (ii) $L-K_M$ is ample, then by Skoda division theorem, we have the following: let $s_0,\cdots,s_N$ be an orthonormal basis of $H^0(M;L)$ with respect to the $L^2$ metric, then for any $s\in H^0(M;L^k)$, $k>n+1$, we have the decomposition (cf. \cite[Proposition 7]{CLi})
\begin{equation}
s=\sum_{\alpha\in\mathbb{N}^N,|\alpha|=k-n-1} e_\alpha s_0^{\alpha_0}\cdots s_N^{\alpha_N}
\end{equation}
where $e_\alpha\in H^0(M;L^{n+1})$ satisfying
\begin{equation}
\int_M|e_\alpha|^2_{h^{n+1}}\omega^n\le C(k,h,\omega)\int_M|s|_{h^k}^2\omega^n,
\end{equation}
the constant $C(k,h,\omega)$ depends on the metric $h$, $\omega$, the power $k$ and the upper and lower bound of the Bergman kernel $\rho_0(x)=\sum_i|s_i(x)|_h^2$.



\subsection{Cheeger-Gromov convergence: global convergence}

From now on, let $\omega_t, t\in[0,T),$ be the maximal solution to the continuity equation (\ref{M-A: continuity}) on a K\"ahler manifold $M$. The metric $\omega_t$ converges smoothly on the ample locus $M_{\amp}=M\backslash\mathcal{S}_M$. In addition, since $\Ric(\omega_t)\ge-\frac{1}{t}\omega_t$, by Gromov precompactness theorem, for any sequence $t_i\rightarrow T$ and fixed $x_0\in M_{\amp}$, we may assume that
\begin{equation}
(M,\omega_{t_i},x_0)\stackrel{d_{GH}}{\longrightarrow}(M_T,d_T,x_T)
\end{equation}
after passing to a subsequence if necessary. The limit $(M_T,d_T)$ is a complete length metric space, maybe noncompact in a prior. It has a regular/singular decomposition $M_T=\mathcal{R}\cup\mathcal{S}$, a point $x\in \mathcal{R}$ iff the tangent cone at $x$ is the Euclidean space $\mathbb{R}^{2n}$. The proof of the following lemma is exactly same as \cite[Proposition 8]{Ga13} so we omit it.

\begin{lemm}\label{local regularity}
There is a constant $\delta>0$ such that for any $\frac{T}{2}\le t<T$, if a metric ball $B_{\omega_t}(x,r)$ satisfies
\begin{equation}
\vol_{\omega_t}\big(B_{\omega_t}(x,r)\big)\ge(1-\delta)\vol(B_r^0)
\end{equation}
where $\vol(B_r^0)$ is the volume of a metric ball of radius $r$ in $2n$-Euclidean space, then
\begin{equation}
\Ric(\omega_t)\le(2n-1)r^{-2}\omega_t,\,\mbox{ in }B_{\omega_t}(x,\delta r).
\end{equation}
\end{lemm}

If $x\in\mathcal{R}$, then, by Colding's volume convergence theorem \cite{Co97}, there is $r=r(x)>0$ such that $\mathcal{H}^{2n}(B_{d_T}(x,r))\ge(1-\frac{\delta}{2})\vol(B_r^0)$, where $\mathcal{H}^{2n}$ denotes the Hausdorff measure. Suppose $x_i\in M$ satisfying $x_i\stackrel{d_{GH}}{\longrightarrow} x$, then by the volume convergence theorem again, $\vol_{\omega_{t_i}}(B_{\omega_{t_i}}(x_i,r))\ge(1-\delta)\vol(B_r^0)$ for $i$ sufficiently large. According to above lemma, together with Anderson's harmonic radius estimate \cite{An90}, there is $\delta^{'}=\delta^{'}(\alpha)>0$ for any $0<\alpha<1$ such that the $C^{1,\alpha}$ harmonic radius at $x_i$ is bigger than $\delta^{'}\delta r$. Passing to the limit, it gives a $C^{1,\alpha}$ harmonic coordinate on $B_{d_T}(x,r)$. This implies in particular that $B_{d_T}(x,r)\subset\mathcal{R}$. So $\mathcal{R}$ is open with a $C^{1,\alpha}$ K\"ahler metric, denoted by $\bar{\omega}_T$; moreover the metric $\omega_{t_i}$ converges in $C^{1,\alpha}$ topology to $\bar{\omega}_T$ on $\mathcal{R}$ for any $\alpha\in(0,1)$.

For any metric $\omega$ let $d_\omega$ be the length metric induced by $\omega$.

\begin{lemm}\label{metric completion: intrinsic}
$(M_T,d_T)=\overline{(\mathcal{R},d_{\bar{\omega}_T)}}$, the metric completion of $(\mathcal{R},d_{\bar{\omega}_T})$.
\end{lemm}
\begin{proof}
There is an exhaustion of $\mathcal{R}$ by compact subsets $K_i$ with $K_i\subset K_{i+1}$ and a sequence of embeddings $\phi_i:K_i\rightarrow M$ such that (1) $\phi_i(x_T)=x_0$, (2) $\phi_i^*\omega_{t_i}\stackrel{C^{1,\alpha}}{\longrightarrow}\bar{\omega}_T$ on $\mathcal{R}$ and (3) $\phi_i$ defines a Gromov-Hausdorff approximation of the convergence $(M,\omega_{t_i},x_0)\stackrel{d_{GH}}{\longrightarrow}(M_T,d_T,x_T)$. The third fact follows from a standard argument using $\codim\mathcal{S}\ge 2$, cf. \cite{ChCo00}; see also \cite{RoZh11} or \cite{ZhZL10}. Moreover, (3) together with (1) implies $(M,\omega_{t_i},x_0)$ converges in the Gromov-Hausdorff topology to $\overline{(\mathcal{R},d_{\bar{\omega}_T)}}$. By the uniqueness of the complete limit space we have $(M_T,d_T)=\overline{(\mathcal{R},d_{\bar{\omega}_T)}}$.
\end{proof}

\begin{lemm}
$\mathcal{R}$ is geodesically convex in $M_T$ in the sense that any minimal geodesic with endpoints in $\mathcal{R}$ lies in $\mathcal{R}$.
\end{lemm}
\begin{proof}
It is simply a consequence of Colding-Naber's H\"older continuity of tangent cones along a geodesic in $M_T$ \cite{CoNa12}. Actually, if $x,y\in\mathcal{R}$, then for any minimal geodesic connecting $x$ and $y$ a neighborhood of endpoints lies in $\mathcal{R}$, so the geodesic will never touch the singular set $\mathcal{S}$.
\end{proof}

Let $D$ be any divisor such that $\mathcal{S}_M\subset D$. Define the Gromov-Hausdorff limit of $D$
$$D_T=:\{x\in M_T|\mbox{ there exists }x_i\in D \mbox{ such that }x_i\stackrel{d_{GH}}{\longrightarrow}x.\}$$

\begin{prop}\label{Cheeger-Gromov convergence: prop 1}
$(M_T,d_T)$ is isometric to $\overline{(M\backslash D,d_{\omega_T})}$.
\end{prop}
\begin{proof}
First observe that, by the smooth convergence outside $D$, $M_T\backslash D_T\subset\mathcal{R}$ and $(M_T\backslash D_T,\bar{\omega}_T)$ is isometric to $(M\backslash D,\omega_T)$. We make the following

\begin{clai}
$D_T\backslash\mathcal{S}\subset\mathcal{R}$ is a subvariety of dimension $(n-1)$ if it is not empty.
\end{clai}
\begin{proof}[Proof of the Claim]
Let $x\in D_T\backslash\mathcal{S}$ and $x_i\in T$ converges to $x$. By the $C^{1,\alpha}$ convergence of $\omega_{t_i}$ around $x$, there are $C,r>0$ independent of $i$ and a sequence of harmonic coordinates in $B_{\omega_{t_i}}(x_i,r)$ such that $C^{-1}\omega_E\le\omega_{t_i}\le C\omega_E$ where $\omega_E$ is the Euclidean metric in the coordinates. Since the total volume of $D$ is uniformly bounded for any $\omega_t$, the local analytic sets $D\cap B_{\omega_{t_i}}(x_i,r)$ have a uniform bound of degree and so converge to an analytic set $D_T\cap B_{d_T}(x,r)$.
\end{proof}

It follows the Hausdorff dimension of $D_T=\mathcal{S}\cup(D_T\backslash\mathcal{S})$ is less than or equal to $2n-2$. Then as in Lemma \ref{metric completion: intrinsic}, following the discussion in \cite{ChCo00}, one can show that the length metric $d_{\bar{\omega}_T}$ on $M_T\backslash D_T$ is same as $d_T$. It infers the required isometry
$$(M_T,d_T)=\overline{(M_T\backslash D_T,d_{\bar{\omega}_T})}\stackrel{\iso}{\cong}\overline{(M\backslash D,d_{\omega_T})}.$$
\end{proof}

A direct corollary is

\begin{coro}
$(M,\omega_t,x_0)$ converges globally to $(M_T,d_T,x_T)$ in the Cheeger-Gromov sense as $t\rightarrow T$.
\end{coro}

Let $M_{\sing}$ be the subvariety of critical points of $\Phi$ and $M_{\reg}=M\backslash M_{\sing}$. We have shown that $\omega_T$ is a smooth metric on $M_{\reg}$. Another corollary is

\begin{coro}
$(M_T,d_T)$ is isometric to $\overline{(M_{\reg},d_{\omega_T})}$.
\end{coro}
\begin{proof}
Notice that $M_{\reg}\backslash(M\backslash D)=M_{\reg}\cap D$ has codimension 2 in $(M_{\reg},\omega_T)$. Thus, the length metric $d_{\omega_T}$ on $M\backslash D$ equals to the restricted extrinsic metric from $(M_{\reg}, d_{\omega_T})$. Since $M\backslash D$ is dense in $M_{\reg}$, we conclude the desired result
$$(M_T,d_T)\stackrel{\iso}{\cong}\overline{(M\backslash D,d_{\omega_T})}=\overline{(M_{\reg},d_{\omega_T})|_{M\backslash D}}=\overline{(M_{\reg},d_{\omega_T})}.$$
\end{proof}

\begin{lemm}
The identity map $\id: M_{\reg}\rightarrow M$ gives a Gromov-Hausdorff approximation representing the convergence $(M,\omega_t)\rightarrow (M_T,d_T)$ as $t\rightarrow T$.
\end{lemm}
\begin{proof}
First observe that $(M\backslash D,d_T)$ is dense in $(M_{\reg},d_T)$ and $(M\backslash D,d_T)=(M\backslash D,d_{\omega_T})$. So, it suffices to show that $\id:(M\backslash D,d_{\omega_T})\rightarrow (M,\omega_t)$ defines a Gromov-Hausdorff approximation. This follows from the same argument as in the proof of Lemma \ref{metric completion: intrinsic}.
\end{proof}

Therefore, the identity map $\id$ extends to an isometry
$$\overline{\id}:\overline{(M_{\reg},d_{\omega_T})}\rightarrow (M_T,d_T).$$

\begin{prop}\label{Cheeger-Gromov convergence: prop 2}
{\rm(1)} $\omega_{t}$ converges smoothly on $M_{\reg}$ to $\omega_T$.

{\rm(2)} $\overline{\id}(M_{\reg})=\mathcal{R}$, the regular set of $M_T$.
\end{prop}
\begin{proof}
(1) For any compact subset $K\subset M_{\reg}$, there is $r=r_K>0$ such that $\vol_{\omega_T}\big(B_{d_T}(x,r)\big)\ge(1-\frac{\delta}{2})\vol(B_r^0)$ for any $x\in K$, where $\delta$ is the constant in Lemma \ref{local regularity}. Then, since the identity map represents the Gromov-Hausdorff convergence, we have $\vol_{\omega_{t_i}}\big(B_{\omega_t}(x,r)\big)\ge(1-\delta)\vol(B_r^0)$ for any $x\in K$ and $t$ sufficiently close to $T$. By Lemma \ref{local regularity}, the Ricci curvature $\Ric(\omega_{t})\le C\omega_{t}$ uniformly on $K$ for some constant $C=C(K)$. Together with the uniform $L^\infty$ bound of $u_{t}$, the continuity equation (\ref{M-A: continuity}) then shows
$$C^{-1}\omega_{t_0}\le \omega_{t}\le C\omega_{t_0},\,\mbox{ on }K$$
for some $C=C(K)$ independent of $t<T$. Then by a standard bootstrap to the complex Monge-Amp\`ere equation (\ref{M-A: continuity 2}) we get the uniform $C^k$ estimate of the metrics $\omega_{t}$ on $K$, for any $k\ge 1$. This is sufficient to prove the smooth convergence of $\omega_{t}$ on $K$.

(2) Since $M_{\reg}$ has smooth structure in $M_T$ we have immediately $\overline{\id}(M_{\reg})\subset\mathcal{R}$. Next we show the converse, namely $\overline{\id}(M_{\reg})$ is the maximal regular subset of $M_T$. The idea follows from \cite{RoZh11}; see also \cite{So14}. We argue by contradiction. Suppose we have a point $p\in\mathcal{R}\backslash\overline{\id}(M_{\reg})$, then there exists a family of points $p_t\in M_{\sing}$ such that $p_t\rightarrow p$. Denote $m=\dim_{\mathbb{C}}M_{\sing}$. By $C^{1,\alpha}$ convergence on $\mathcal{R}$, there exist $C,r>0$ independent of $t$ and a sequence of harmonic coordinates on $B_{\omega_{t}}(p_t,r)$ such that $C^{-1}\omega_{E}\le\omega_{t}\le C\omega_E$ where $\omega_E$ is the Euclidean metric in this coordinate. Then
$$\vol_{\omega_{t}}\big(M_{\sing}\cap B_{\omega_{t}}(p_t,r)\big)=\int_{M_{\sing}\cap B_{\omega_{t}}(p_t,r)}\omega_{t}^m\ge \int_{M_{\sing}\cap B_{\omega_E}(C^{-1/2}r)}(C^{-1}\omega_E)^m$$
which has a uniform lower bound $C^{-2m}c(m)r^{2m}$ where $c(m)$ is the volume of unit sphere in $\mathbb{C}^m$. This follows from the classical analysis of the lower volume estimate or multiplicity estimate of an analytical set in the Euclidean space. However, this contradicts with the degeneration of the limit metric $\eta_T$ along $M_{\sing}$:
$$\vol_{\omega_{t}}\big(M_{\sing}\cap B_{\omega_{t}}(p_t,r)\big)\le\vol_{\omega_{t}}(M_{\sing})=\int_{M_{\sing}}\omega_{t}^m
=\int_{M_{\sing}}\eta_{t}^m=\bigg(\frac{T-t}{T}\bigg)^m\int_{M_{\sing}}\omega_0^m$$
which tends to 0 as $t\rightarrow T$. So we have $\overline{\id}(M_{\reg})\supset\mathcal{R}$ as desired.
\end{proof}

It follows that, for any Hermitian line bundle $(L',h')$ and $k\in\mathbb{Z}$, the twisted line bundle $(L'\otimes K_M^k,h'\otimes(\omega_{t_i}^{-n})^k)$ converges smoothly to a limit Hermitian line bundle $(L'\otimes K_{\mathcal{R}}^k,h'\otimes(\bar{\omega}_T^{-n})^k)$ on $\mathcal{R}$. Another corollary is $M_{\amp}\subset M_{\reg}$.



\subsection{$L^2$ estimate to $\bar\partial$-operator}

Let $L=\ell_0(L^{'}+TK_M)$ be the limit line bundle. Up to raising the power $\ell_0$, we assume that (i) $L$ is semi-ample and (ii) $L-K_M$ is ample. The later follows from $L-K_M=\ell_0(L^{'}+tK_M)$ where $L'+tK_M$ is ample because $t=T-\frac{1}{\ell_0}<T$.

Let $\eta_T$, $u_t$ be defined as in the previous section. Choose a Hermitian metric $h_{L^{'}}$ on $L^{'}$ whose curvature form $\Theta_{h_{L^{'}}}=\omega_0$ and put $h_t=h_{L^{'}}^{\ell_0}\otimes \big(\omega_t^{-n}\big)^{\ell_0 T}$, a family of Hermitian metric on $L$ for any $0\le t<T$. The curvature form of $h_t$ is
\begin{equation}\label{curvature: h_t}
\Theta_{h_t}=\ell_0\frac{T}{t}\omega_t-\ell_0\frac{T-t}{t}\omega_0.
\end{equation}


\begin{lemm}
For any $k\ge 1$ and smooth section $\xi\in \Gamma(T^{0,1}\otimes L^k)$ we have
\begin{equation}
\int_M\big(|\bar{\partial}\xi|^2+|\bar{\partial}^*\xi|^2\big)\omega_t^n\ge\frac{k\ell_0 T-1}{t}\int_M|\xi|^2\omega_t^n,\,\forall t\in[T-\frac{1}{k\ell_0},T).
\end{equation}
\end{lemm}
\begin{proof}
Combining with the curvature formulas (\ref{curvature: h_t}) and $\Ric(\omega_t)=\frac{\omega_0-\omega_t}{t}$ derives
$$\Theta_{h_t^k}+\Ric(\omega_t)=\frac{\omega_0}{t}\big(1-k\ell_0(t-T)\big)+\frac{\omega_t}{t}(k\ell_0 T-1).$$
Then apply the Weitzenb\"och formula (\ref{Weitzenboch formula: 01}).
\end{proof}

\begin{prop}[$L^2$ estimate]\label{L^2 estimate: prop 1}
For any $k\ge\frac{2}{\ell_0 T}$ and $t\in[T-\frac{1}{k\ell_0},T)$ and $\xi\in C^\infty(M,T^{1,0}M\otimes L^k)$ with $\bar{\partial}\xi=0$ we can find a solution $\bar{\partial}\varsigma=\xi$ which satisfies
\begin{equation}\label{L^2 estimate: 1}
\int_M|\varsigma|_{h_t^k}^2\omega_t^n\le\frac{2}{k}\int_M|\xi|_{h_t^k\otimes\omega_t}^2\omega_t^n.
\end{equation}
\end{prop}
\begin{proof}
By above corollary, the Hodge Laplacian $\triangle_{\bar{\partial}}$ on $T^{1,0}M\otimes L^k$ is strictly positive, in fact $\triangle_{\bar{\partial}}\ge\frac{k\ell_0 T-1}{t}\ge\frac{k}{2}$, when $k\ge\frac{2}{\ell_0T}$. This in turn implies the first positive eigenvalue of $\triangle_{\bar{\partial}}$ on $L$ is bigger than $\frac{k}{2}$. The solvability of $\bar{\partial}\varsigma=\xi$ and the $L^2$ estimate of the solution are easy from classical Hodge theory.
\end{proof}



\subsection{$L^\infty$ estimate to holomorphic sections}

Recall that the curvature of $h_t$
$$\Theta_{h_t}=\ell_0\frac{T}{t}\omega_t-\ell_0\frac{T-t}{t}\omega_0\le\ell_0\frac{T}{t}\omega_t.$$
So, by the Bochner formula (\ref{Bochner formula: 01}) we have
\begin{equation}
\triangle_{\omega_t}|\varsigma|_{h_t^k}^2\ge |\nabla\varsigma|_{h_t^k\omega_t}^2-k\ell_0\frac{T}{t}|\varsigma|_{h_t^k}^2,\,\forall\varsigma\in H^0(M;L^k).
\end{equation}

Also recall that we have the following well-known Sobolev inequality: for any $R>0$, there is $C(R)$ independent of $t$ such that
\begin{equation}\nonumber
\bigg(\int_{B_{\omega_t}(x_0,R)}f^{\frac{2n}{n-2}}\omega_t^n\bigg)^{\frac{n-1}{n}}\le C(R)\int_{B_{\omega_t}(x_0,R)}\big(f^2+|\nabla f|^2_{\omega_t}\big)\omega_t^n,
\end{equation}
for all $f\in C_0^1\big(B_{\omega_t}(x_0,R)\big)$.

By a standard iteration argument we have

\begin{lemm}\label{infinity estimate}
For any $R>0$, there exists $C(R)$ independent of $t$ and $k\ge 1$ such that for any $\frac{T}{2}\le t<T$ and $B_{\omega_t}(x,2r)\subset B_{\omega_t}(x_0,R)$, if $\varsigma\in H^0(B_{\omega_t}(x,2r);L^k)$, then
\begin{equation}
\sup_{B_{\omega_t}(x,r)}|\varsigma|_{h_t^k}^2\le C(R)\cdot r^{-2n}\cdot k^n\int_{B_{\omega_t}(x,2r)}|\varsigma|_{h_t^k}^2\omega_t^n.
\end{equation}
\end{lemm}

Recall the Cheeger-Gromov convergence
\begin{equation}
(M,\omega_{t},x_0)\stackrel{d_{GH}}{\longrightarrow}(M_T,d_T,x_T).
\end{equation}
Define the Hermitian line bundle $(L_T,h_T)$ on the regular set $\mathcal{R}\subset M_T$ by
$$L_T=\ell_0(L'+TK_{\mathcal{R}}),\,h_T=h_{L'}^{\ell_0}\otimes\bar{\omega}_T^{-n\ell_0T}.$$
Under the isometry $\overline{\id}:\overline{(M_{\reg},d_{\omega_T})}\rightarrow(M_T,d_T)$ constructed in Subsection 3.2, the Hermitian line bundle is isometric to $(L,h_{L'}^{\ell_0}\otimes\omega_T^{-n\ell_0T})$ on $M_{\reg}$.
Furthermore, the Hermitian line bundles $(L,h_{t})$ converges smoothly to $(L_T,h_T)$ on $\mathcal{R}$.

\begin{coro}
Let $R>0$ be any constant, $t_i\rightarrow T$ be any subsequence of times and $\varsigma_i$ be a sequence of holomorphic sections of $L^k$, $k\ge 1$, satisfying
\begin{equation}
\int_M|\varsigma_i|_{h_{t_i}^k}^2\omega_{t_i}^n\le 1.
\end{equation}
Then, passing to a subsequence if necessary, $\varsigma_i$ converges to a locally bounded holomorphic section $\varsigma_\infty$ of $L_T^k$ over $\mathcal{R}$ which satisfies
\begin{equation}
\sup_{B_{d_T}(x,r)\cap\mathcal{R}}|\varsigma_\infty|_{h_T^k}^2\le C(R)\cdot r^{-2n}\cdot k^n\int_{B_{d_T}(x,2r)\cap\mathcal{R}}|\varsigma_\infty|_{h_T^k}^2\bar{\omega}_T^n.
\end{equation}
whenever $B_{d_T}(x,r)\subset B_{d_T}(x_T,R)$.
\end{coro}



\subsection{Gradient estimate to holomorphic sections: a prior $L^\infty$ finiteness}

Next we derive the gradient estimate to holomorphic sections of $L^k$. Recall that by the Bochner formula (\ref{Bochner formula: 02}) we get, for any holomorphic section $\varsigma\in H^0(M;L^k)$ and time $t<T$,
\begin{eqnarray}
\triangle_{\omega_t}|\nabla\varsigma|_{h_t^k\otimes\omega_t}^2\ge && |\bar{\nabla}^{h_t}\nabla^{h_t}\varsigma|_{h_t^k\otimes\omega_t}^2
+|\nabla^{h_t}\nabla^{h_t}\varsigma|_{h_t^k\otimes\omega_t}^2-
\frac{(n+2)k\ell_0 T+1}{t}|\nabla\varsigma|_{h_t^k\otimes\omega_t}^2 \nonumber\\
  && +k\ell_0(T-t)\re\big(\nabla_is\langle\varsigma,\nabla_{\bar{i}}\bar{\varsigma}\rangle_{h_t^k\otimes\omega_t}
  \big)\label{Bochner formula: 03}
\end{eqnarray}
where $s$ is the scalar curvature of $\omega_t$, all terms are calculated with respect to $\omega_t$ and $h_t^k$. To verify this, just notice that $\Ric=\frac{1}{t}(\omega_0-\omega_t)$ and
$$-\nabla_j\Theta_{i\bar{j}}=k\ell_0(T-t)\nabla_jR_{i\bar{j}}=k\ell_0(T-t)\nabla_is.$$
In order to derive a uniform gradient estimate from the formula (\ref{Bochner formula: 2}) we need a Type I estimate to the scalar curvature under the continuity equation. Motivated by the K\"ahler-Ricci flow this should be true in general. However, as for K\"ahler-Ricci flow, it is difficult to show the Type I property right now.

Another approach is to consider the gradient estimate to the limit sections of $L_T^k$ on $\mathcal{R}$. However, a prior $L^\infty$ bound of holomorphic sections will be necessary for the limit process. To do this from now on in this subsection we introduce a new Hermitian metric on $L$, namely,
\begin{equation}
h_{FS}=h_{L'}^{\ell_0}\otimes\Omega^{-\ell_0T}
\end{equation}
where $h_{L'}$ is the Hermitian metric on $L'$ whose curvature $\Theta_{h_{L'}}=\omega_0$, $\Omega$ is the volume form on $M$ whose curvature $\Theta_{\Omega}=\frac{1}{T}(\omega_0-\eta_T)$. The metric $h_{FS}$ has curvature
\begin{equation}
\Theta_{h_{FS}}=\ell_0\eta_T,
\end{equation}
where $\eta_T$ is the induced Fubini-Study metric which satisfies $\eta_T\le C\omega_t$ for some $C$ independent of $t$; see Section 2. An easy calculation shows that for any $0<t<T$,
\begin{equation}
h_t=e^{-\ell_0\frac{T}{t}u_t}h_{FS},
\end{equation}
so $h_{FS}$ is uniformly equivalent to $h_t$ for any $\frac{T}{2}\le t<T$.

In the following computation we denote $\nabla\varsigma=\nabla^{h_{FS}^k}\varsigma$, $\bar{\nabla}\nabla\varsigma=\bar{\nabla}^{h_{FS}^k}\nabla^{h_{FS}^k}\varsigma$ and $|\nabla\varsigma|=|\nabla^{h_{FS}^k}\varsigma|_{h_{FS}^k\otimes\omega_t}$, etc., for any $\varsigma\in H^0(M;L^k)$, $k\ge 1$.

\begin{lemm}
For any $\frac{T}{2}\le t<T$ and $\varsigma\in H^0(M,L^k)$, $k\ge 1$, we have
\begin{equation}\label{Bochner formula: 1}
\triangle|\varsigma|^2\ge|\nabla\varsigma|^2-Ck|\varsigma|^2,
\end{equation}
and
\begin{equation}\label{Bochner formula: 2}
\triangle|\nabla\varsigma|^2
\ge|\bar{\nabla}\nabla\varsigma|^2+|\nabla\nabla\varsigma|^2
-k\ell_0\nabla_j(\eta_T)_{i\bar{j}}\langle\varsigma,\nabla_{\bar{i}}\bar{\varsigma}\rangle
-k\ell_0\nabla_{\bar{j}}(\tr_{\omega}\eta_T)\langle\nabla_j\varsigma,\bar{\varsigma}\rangle
-Ck|\nabla\varsigma|^2,
\end{equation}
where $C$ does not depend on $t$, $k$ and $\varsigma$.
\end{lemm}
\begin{proof}
(\ref{Bochner formula: 1}) is a direct consequence of the Bochner formula (\ref{Bochner formula: 01}). (\ref{Bochner formula: 2}) follows from the Bochner formula (\ref{Bochner formula: 02})
\begin{eqnarray}
\triangle|\nabla\varsigma|^2&=&|\bar{\nabla}\nabla\varsigma|^2+|\nabla\nabla\varsigma|^2
-k\ell_0\nabla_j(\eta_T)_{i\bar{j}}\langle\varsigma,\nabla_{\bar{i}}\bar{\varsigma}\rangle
-k\ell_0\nabla_{\bar{j}}(\tr_{\omega}\eta_T)\langle\nabla_j\varsigma,\bar{\varsigma}\rangle\nonumber\\
&&\hspace{1cm}+R_{i\bar{j}}\langle\nabla_j\varsigma,\nabla_{\bar{i}}\bar{\varsigma}\rangle
-2k\ell_0(\eta_T)_{i\bar{j}}\langle\nabla_j\varsigma,\nabla_{\bar{i}}\bar{\varsigma}\rangle
-k\ell_0|\nabla\varsigma|^2\cdot\tr_{\omega}\eta_T\nonumber.
\end{eqnarray}
and $\Ric(\omega_t)=\frac{1}{t}(\omega_0-\omega_t)\ge-\frac{1}{t}\omega_t$.
\end{proof}

\begin{prop}\label{gradient estimate: prop 1}
For any $R>0$, there exists $C(R)$ independent of $t$ and $k\ge 1$ such that for any $\frac{T}{2}\le t<T$ and $B_{\omega_t}(x,2r)\subset B_{\omega_t}(x_0,R)$, if $\varsigma\in H^0(B_{\omega_t}(x,2r);L^k)$, then
\begin{equation}
\sup_{B_{\omega_t}(x,r)}|\varsigma|_{h_{FS}^k}^2\le C(R)\cdot r^{-2n}\cdot k^n\int_{B_{\omega_t}(x,2r)}|\varsigma|_{h_{FS}^k}^2\omega_t^n.
\end{equation}
\begin{equation}\label{gradient estimate: 2}
\sup_{B_{\omega_t}(x,r)}|\nabla^{h_{FS}}\varsigma|^2_{h_{FS}^k\otimes\omega_t}\le C(R)\cdot r^{-2n-2}\cdot k^{n+1}\int_{B_{\omega_t}(x,2r)}|\varsigma|_{h_{FS}^k}^2\omega_t^n.
\end{equation}
\end{prop}

\begin{proof}
The proof uses Nash-Moser iteration. We only prove the gradient estimate, the $L^\infty$ estimate is obvious by (\ref{Bochner formula: 1}). For simplicity we may assume
$$\int_{B_{\omega_t}(x,2r)}|\varsigma|^2\omega_t^n=1.$$
Then the $L^\infty$ estimate of $\varsigma$ shows
$$\sup_{B_{\omega_t}(x,\frac{7}{4}r)}|\varsigma|\le C(R)\cdot r^{-n}k^{\frac{n}{2}}.$$

Keep in mind that $\eta_T\le C\omega_t$ uniformly for any $\frac{T}{2}\le t<T$. Now, for any $p\ge\frac{n}{n-1}$ and cut-off $\rho\in C_0^\infty(B_{\omega_t}(x,2r))$ with $0\le\rho\le 1$, we have by (\ref{Bochner formula: 2})
\begin{eqnarray}
\int_M\rho^2\big|\nabla|\nabla\varsigma|^p\big|^2&=&\frac{p^2}{4(p-1)}
\int\rho^2\nabla_i|\nabla\varsigma|^{2(p-1)}\nabla_{\bar{i}}|\nabla\varsigma|^2\nonumber\\
&=&\frac{p^2}{4(p-1)}
\int\bigg(-\rho^2|\nabla\varsigma|^{2(p-1)}\triangle|\nabla\varsigma|^2
-2\rho|\nabla\varsigma|^{2(p-1)}\nabla_i\rho\nabla_{\bar{i}}|\nabla\varsigma|^2\bigg)\nonumber\\
&\le&\frac{p^2}{4(p-1)}
\int\bigg(-\rho^2|\nabla\varsigma|^{2(p-1)}\big(|\bar{\nabla}\nabla\varsigma|^2
+|\nabla\nabla\varsigma|^2\big)\nonumber\\
&&+k\ell_0\rho^2|\nabla\varsigma|^{2(p-1)}\big(\nabla_j(\eta_T)_{i\bar{j}}
\langle\varsigma,\nabla_{\bar{i}}\bar{\varsigma}\rangle
+\nabla_{\bar{j}}(\tr_{\omega}\eta_T)\langle\nabla_j\varsigma,\bar{\varsigma}\rangle\big)\nonumber\\
&&+Ck\rho^2|\nabla\varsigma|^{2p}
-2\rho|\nabla\varsigma|^{2(p-1)}\nabla_i\rho\nabla_{\bar{i}}|\nabla\varsigma|^2\bigg)\nonumber.
\end{eqnarray}
The term $\int\rho^2|\nabla\varsigma|^{2(p-1)}\nabla_j(\eta_T)_{i\bar{j}}\langle\varsigma,\nabla_{\bar{i}}\bar{\varsigma}\rangle$ can be estimated by integration by parts as follows
\begin{eqnarray}
&&\int k\ell_0\rho^2|\nabla\varsigma|^{2(p-1)}\nabla_j(\eta_T)_{i\bar{j}}
\langle\varsigma,\nabla_{\bar{i}}\bar{\varsigma}\rangle\nonumber\\
&=&-k\ell_0\int(\eta_T)_{i\bar{j}}\nabla_j\big(\rho^2|\nabla\varsigma|^{2(p-1)}
\langle\varsigma,\nabla_{\bar{i}}\bar{\varsigma}\rangle\big)\nonumber\\
&=&-k\ell_0\int(\eta_T)_{i\bar{j}}\bigg(\rho^2|\nabla\varsigma|^{2(p-1)}
\big(\langle\nabla_j\varsigma,\nabla_{\bar{i}}\bar{\varsigma}\rangle
+\langle\varsigma,\nabla_j\nabla_{\bar{i}}\bar{\varsigma}\rangle\big)\nonumber\\
&&+(p-1)\rho^2|\nabla\varsigma|^{2(p-2)}
\nabla_j|\nabla\varsigma|^2\langle\varsigma,\nabla_{\bar{i}}\bar{\varsigma}\rangle+2\rho|\nabla\varsigma|^{2(p-1)}
\nabla_j\rho\langle\varsigma,\nabla_{\bar{i}}\bar{\varsigma}\rangle\bigg)\nonumber\\
&\le&\frac{1}{8}\int\rho^2|\nabla\varsigma|^{2(p-1)}\big(|\bar{\nabla}\nabla\varsigma|^2+|\nabla\nabla\varsigma|^2\big)
+C(p-1)^2k^2\int\rho^2|\varsigma|^2|\nabla\varsigma|^{2(p-1)}\nonumber\\
&&\hspace{2cm}+Ck\int\bigg(|\nabla\rho|^2|\varsigma|^2
|\nabla\varsigma|^{2(p-1)}+\rho^2|\nabla\varsigma|^{2p}\bigg);\nonumber
\end{eqnarray}
similar estimate remains valid for $\int k\ell_0\rho^2|\nabla\varsigma|^{2(p-1)}\nabla_{\bar{j}}(\tr_{\omega}\eta_T)
\langle\nabla_j\varsigma,\bar{\varsigma}\rangle$; moreover,
\begin{eqnarray}
&&-2\int\rho|\nabla\varsigma|^{2(p-1)}\nabla_i\rho\nabla_{\bar{i}}|\nabla\varsigma|^2\nonumber\\
&\le&\frac{1}{4}\int\rho^2|\nabla\varsigma|^{2(p-1)}\big(|\bar{\nabla}\nabla\varsigma|^2
+|\nabla\nabla\varsigma|^2\big)+\int|\nabla\rho|^2|\nabla\varsigma|^{2p}.\nonumber
\end{eqnarray}
Summing up the estimates we get
\begin{eqnarray}
\int\rho^2\big|\nabla|\nabla\varsigma|^p\big|^2&\le& Cp^3k\int\bigg(\rho^2|\nabla\varsigma|^{2p}
+|\nabla\rho|^2|\varsigma|^2|\nabla\varsigma|^{2(p-1)}\nonumber\\
&&+|\nabla\rho|^2|\nabla\varsigma|^{2p}+k\rho^2|\varsigma|^2|\nabla\varsigma|^{2(p-1)}\bigg).\nonumber
\end{eqnarray}
Applying the Sobolev inequality we have
\begin{eqnarray}
\bigg(\int(\rho|\nabla\varsigma|^p)^{\frac{2n}{n-1}}\bigg)^{\frac{n-1}{n}}&\le& Cp^3k\int\bigg(\rho^2|\nabla\varsigma|^{2p}
+|\nabla\rho|^2|\varsigma|^2|\nabla\varsigma|^{2(p-1)}\nonumber\\
&&+|\nabla\rho|^2|\nabla\varsigma|^{2p}+k\rho^2|\varsigma|^2|\nabla\varsigma|^{2(p-1)}\bigg),\nonumber
\end{eqnarray}
where $C=C(R)$. We next consider three independent cases to run the iteration. Put $p_j=\nu^{j+1}$, $j\ge 0$, where $\nu=\frac{n}{n-1}$. Define a family of radius inductively by $r_0=\frac{3}{2}r$ and $r_j=r_{j-1}-2^{-j-1}r$, and a family of cut-offs $\rho_j\in C_0^\infty(B_{j})$ where $B_j=B_{\omega_t}(x_0,r_j)$ such that
$$0\le\rho_j\le 1,\,|\nabla\rho_j|\le2^{j+2}r^{-1}\mbox{ and }\rho_j=1\mbox{ on }B_{j+1}.$$
Above formula finally gives, by setting $\rho=\rho_j$,
\begin{equation}
\bigg(\int_{B_{j+1}}|\nabla\varsigma|^{2p_{j+1}}\bigg)^{\frac{n-1}{n}}\le Cp_j^32^{2j}k r^{-2}\int_{B_{j}}\big(|\nabla\varsigma|^{2p_j}
+k|\varsigma|^2|\nabla\varsigma|^{2(p_j-1)}\big),\,\forall j\ge 0.\label{sobolev}
\end{equation}

Case 1: $\big(\int_{B_{j}}|\nabla\varsigma|^{2p_j}\big)^{\frac{1}{p_j}}\ge k\big(\int_{B_{j}}
|\varsigma|^{2p_j}\big)^{\frac{1}{p_j}}$, for all $j\ge 0$. Then,
$$\int_{B_{j}}k|\varsigma|^2|\nabla\varsigma|^{2(p_j-1)}\le k\bigg(\int_{B_{j}}
|\nabla\varsigma|^{2p_j}\bigg)^{1-\frac{1}{p_j}}\bigg(\int_{B_{j}}
|\varsigma|^{2p_j}\bigg)^{\frac{1}{p_j}}\le\int_{B_{j}}|\nabla\varsigma|^{2p_j}.$$
Then (\ref{sobolev}) gives
$$\bigg(\int_{B_{j+1}}|\nabla\varsigma|^{2p_{j+1}}\bigg)^{\frac{1}{p_{j+1}}}\le (Ckr^{-2})^{\frac{1}{p_j}}p_j^{\frac{3}{p_j}}4^{\frac{j}{p_j}}
\bigg(\int_{B_{j}}|\nabla\varsigma|^{2p_j}\bigg)^{\frac{1}{p_j}},\,\forall j\ge 0.$$
By an easy iteration,
$$\big\||\nabla\varsigma|^2\big\|_{L^{p_j}(B_j)}\le(Ckr^{-2})^{\sum\nu^{-i}}\cdot\Pi p_i^{\frac{3}{p_i}}\cdot 4^{\sum i\nu^{-i}}\cdot\big\||\nabla\varsigma|^2\big\|_{L^{p_0}(B_0)},\,\forall j\ge 0.$$
In particular by letting $j\rightarrow\infty$ we obtain
\begin{eqnarray}
\sup_{B_{\omega_t}(x,r)} |\nabla\varsigma|^2&\le& (Ckr^{-2})^{n-1}\cdot C(n)\cdot\bigg(\int_{B_0}|\nabla\varsigma|^{2\nu}\bigg)^{\frac{1}{\nu}}\nonumber
\end{eqnarray}
Finally, a cut-off argument gives
$$\bigg(\int_{B_0}|\nabla\varsigma|^{2\nu}\bigg)^{\frac{1}{\nu}}\le C(R)\int_{B_{\omega_t}(x,\frac{7}{4}r)\backslash D}\big(|\bar{\nabla}\nabla\varsigma|^2+|\nabla\nabla\varsigma|^2+r^{-2}|\nabla\varsigma|^2\big)\le C(R)k^{-2}r^{-4}.$$
In the last inequality we used the lemma below. So we have
$$\sup_{B_{\omega_t}(x,r)}|\nabla\varsigma|^2\le C(R)(kr^{-2})^{n+1}.$$

Case 2: There exists $j_0\ge 0$ such that $\big(\int_{B_j}|\nabla\varsigma|^{2p_j}\big)^{\frac{1}{p_j}}\ge k\big(\int_{B_j}
|\varsigma|^{2p_j}\big)^{\frac{1}{p_j}}$ for all $j>j_0$ but
$$\big(\int_{B_{j_0}}|\nabla\varsigma|^{2p_{j_0}}\big)^{\frac{1}{p_{j_0}}}<k\big(\int_{B_{j_0}}
|\varsigma|^{2p_{j_0}}\big)^{\frac{1}{p_{j_0}}}.$$
Then,
$$\int_{B_{j_0}}k|\varsigma|^2|\nabla\varsigma|^{2(p_{j_0}-1)}\le k^{p_{j_0}}\int_{B_{j_0}}|\varsigma|^{2p_{j_0}}.$$
By a same iteration for $j\ge j_0+1$ as above we have,
\begin{eqnarray}
\sup_{B_{\omega_t}(x,r)\backslash D}|\nabla\varsigma|^2&\le& (Ckr^{-2})^{\frac{1}{p_{j_0+1}}\sum\nu^{-i}}\cdot C(n)\cdot\bigg(\int_{B_{j_0+1}}|\nabla\varsigma|^{2p_{j_0+1}}\bigg)^{\frac{1}{p_{j_0+1}}}\nonumber\\
&\le&(Ckr^{-2})^{\frac{1}{p_{j_0}}\sum\nu^{-i}}
\cdot\bigg(\int_{B_{j_0}}|\nabla\varsigma|^{2p_{j_0}}+k|\varsigma|^2|\nabla\varsigma|^{2(p_{j_0}-1)}
\bigg)^{\frac{1}{p_{j_0}}}\nonumber\\
&\le&(Ckr^{-2})^{\frac{n}{p_{j_0}}}\cdot
k\cdot\bigg(\int_{B_{j_0}}|\varsigma|^{2p_{j_0}}\bigg)^{\frac{1}{p_{j_0}}}\nonumber.
\end{eqnarray}
The supremum estimate of $|\nabla\varsigma|$ follows from
$$\bigg(\int_{B_{j_0}}|\varsigma|^{2p_{j_0}}\bigg)^{\frac{1}{p_{j_0}}}\le\big(\sup_{B_{j_0}}
|\varsigma|\big)^{\frac{2(p_{j_0}-1)}{p_{j_0}}}\bigg(\int_{B_{j_0}}|\varsigma|^{2}\bigg)^{\frac{1}{p_{j_0}}}
\le(C kr^{-2})^{n-\frac{n}{p_{j_0}}}.$$

Case 3: $\big(\int_{B_j}|\nabla\varsigma|^{2p_j}\big)^{\frac{1}{p_j}}\le k\big(\int_{B_j}
|\varsigma|^{2p_j}\big)^{\frac{1}{p_j}}$ for all $j\ge 0$. It is obvious that
$$\sup_{B_{\omega_t}(x,r)}|\nabla\varsigma|\le k\sup_{B_{\omega_t}(x,r)}|\varsigma|\le Ck^{n+1}r^{-2n}.$$

Summing up the three cases we conclude the gradient estimate on $B_{\omega_t}(x,r)$.
\end{proof}

\begin{lemm}
Assume as in above proposition. We have
\begin{equation}\label{Bochner formula: 4}
\int_{B_{\omega_t}(x,\frac{7}{4}r)}|\nabla\varsigma|^2\omega_t^n\le Ckr^{-2}\int_{B_{\omega_t}(x,2r)}|\varsigma|^2\omega_t^n,
\end{equation}
and
\begin{equation}
\int_{B_{\omega_t}(x,\frac{7}{4}r)}\big(|\bar{\nabla}\nabla\varsigma|^2+|\nabla\nabla\varsigma|^2\big)\omega_t^n\le Ck^2r^{-4}\int_{B_{\omega_t}(x,2r)}|\varsigma|_{h_t}^2\omega_t^n.
\end{equation}
where $C$ is a constant independent of $t$, $k$ and $\varsigma\in H^0(B_{\omega_t}(x,2r),L^k)$.
\end{lemm}
\begin{proof}
Let $\rho\in C_0^\infty(B_{\omega_t}(x,\frac{15}{8}r))$ be any cut-off such that $0\le\rho\le 1$, $|\nabla\rho|\le 100r^{-2}$ and $\rho=1$ on $B_{\omega_t}(x,\frac{7}{4}r)$, then by (\ref{Bochner formula: 1}),
$$\int\rho^2|\nabla\varsigma|^2\le Ck\int\rho^2|\varsigma|^2+\int\rho^2\triangle|\varsigma|^2$$
where
$$\int\rho^2\triangle|\varsigma|^2=-2\int\rho\nabla_{\bar{i}}\rho\langle\nabla_i\varsigma,
\bar{\varsigma}\rangle
\le\frac{1}{2}\int\rho^2|\nabla\varsigma|^2+2\int|\nabla\varsigma|^2|\varsigma|^2.$$
Thus,
$$\int\rho^2|\nabla\varsigma|^2\le Ckr^{-2}\int_{B_{\omega_t}(x,\frac{15}{8}r)}|\varsigma|^2.$$
The first estimate follows. Then, by (\ref{Bochner formula: 2}),
\begin{eqnarray}
&&\int\rho^2\big(|\bar{\nabla}\nabla\varsigma|^2+|\nabla\nabla\varsigma|^2\big)\nonumber\\
&\le&\int\rho^2
\bigg(\triangle|\nabla\varsigma|^2+k\ell_0\nabla_j(\eta_T)_{i\bar{j}}\langle\varsigma,
\nabla_{\bar{i}}\bar{\varsigma}\rangle
+k\ell_0\nabla_{\bar{j}}(\tr_{\omega}\eta_T)\langle\nabla_j\varsigma,\bar{\varsigma}\rangle\bigg)
+Ck\int\rho^2|\nabla\varsigma|^2,\nonumber
\end{eqnarray}
where
$$\int\rho^2\triangle|\nabla\varsigma|^2=-2\int\rho\nabla_i\rho\nabla_{\bar{i}}|\nabla\varsigma|^2
\le\frac{1}{4}\int\rho^2\big(|\bar{\nabla}\nabla\varsigma|^2+|\nabla\nabla\varsigma|^2\big)
+8\int|\nabla\rho|^2|\nabla\varsigma|^2,$$
\begin{eqnarray}
\int k\ell_0\rho^2\nabla_j(\eta_T)_{i\bar{j}}\chi^2\langle\varsigma,\nabla_{\bar{i}}\bar{\varsigma}\rangle
&=&-\int k\ell_0\rho^2(\eta_T)_{i\bar{j}}\big(\langle\nabla_j\varsigma,\nabla_{\bar{i}}\bar{\varsigma}\rangle
+\langle\varsigma,\nabla_j\nabla_{\bar{i}}\bar{\varsigma}\rangle\big)\nonumber\\
&&-2k\ell_0\int\rho(\eta_T)_{i\bar{j}}\nabla_j\rho\langle\varsigma,\nabla_{\bar{i}}\bar{\varsigma}\rangle\nonumber\\
&\le&\frac{1}{4}\int\rho^2|\bar{\nabla}\nabla\varsigma|^2+Ck^2\int\rho^2|\varsigma|^2
+C\int|\nabla\rho|^2|\nabla\varsigma|^2,\nonumber
\end{eqnarray}
the estimate to the integrand $\nabla_{\bar{j}}(\tr_{\omega}\eta_T)\langle\nabla_j\varsigma,\bar{\varsigma}\rangle$ has the same form. Summing up these we get
$$\int_{B_{\omega_t}(x,\frac{7}{4}r)}\big(|\bar{\nabla}\nabla\varsigma|^2+|\nabla\nabla\varsigma|^2\big)
\le Ckr^{-2}\int_{B_{\omega_t}(x,\frac{15}{8}r)}|\nabla\varsigma|^2+Ck^2\int_{B_{\omega_t}(x,2r)}|\varsigma|^2.$$
Then use the first estimate once again on $B_{\omega_t}(x,\frac{15}{8}r)$ to get the second estimate.
\end{proof}



\subsection{Gradient estimate to holomorphic sections: estimate to the limit sections}

Notice that the Hermitian metric $h_{FS}$ is equivalent to $h_t$. So, by the gradient estimate  in Proposition \ref{gradient estimate: prop 1} we have that any holomorphic section $\varsigma\in H^0(M;L^k)$ satisfies a uniform $L^\infty$ gradient bound
\begin{equation}
\sup_{B_{\omega_t}(x_0,R)}|\nabla^{h_t}\varsigma|^2_{h_t^k\otimes\omega_t}\le C(R,k)\int_{B_{\omega_t}(x_0,2R)}|\varsigma|_{h_t^k}^2\omega_t^n
\end{equation}
where $x_0\in M\backslash\mathcal{S}_M$ is the base point in the Cheeger-Gromov convergence, $C(R,k)$ is a constant depending on $R$ and $k$ but independent of $t$ and $\varsigma$. This type of estimate is not good enough for constructing local peak sections as the power $k$ becomes sufficiently large. A better estimate can be proved based on the following technical lemma.

\begin{lemm}
There is a family of cut-offs $\gamma_\epsilon\in C_0^\infty(\mathcal{R})$, $\epsilon>0$, with $0\le \gamma_\epsilon\le 1$ such that $\gamma_\epsilon^{-1}(1)$ forms an exhaustion of $\mathcal{R}$
and, moreover,
\begin{equation}
\int_{M_T}|\bar{\partial}\gamma_\epsilon|^2\bar{\omega}_T^n\rightarrow 0,\,\mbox{ as }\epsilon\rightarrow 0.
\end{equation}
\end{lemm}
\begin{proof}
Notice that by Proposition \ref{Cheeger-Gromov convergence: prop 2}, the regular set $(\mathcal{R},\bar{\omega}_T)$ can be identified with $(M_{\reg},\omega_T)$ where $M_{\reg}$ is the regular set of the holomorphic map $\Phi:M\rightarrow\mathbb{C}P^N$. The metric $\omega_T=\eta_T+\sqrt{-1}\partial\bar{\partial}u_T$ where $u_T$ is a bounded potential. Then the lemma is more or less standard in the pluripotential theory. For a detailed proof we refer to \cite[Lemma 6.4]{Ti12}; see also \cite{Ti14} or \cite[Lemma 3.7]{So14}.
\end{proof}

By a standard iteration we have (cf. \cite{So14})

\begin{prop}\label{gradient estimate: prop 2}
Let $R>0$ be any constant, $t_i\rightarrow T$ be any subsequence of times and $\varsigma_i$ be a sequence of holomorphic sections of $L^k$, $k\ge 1$, satisfying
\begin{equation}
\int_M|\varsigma_i|_{h_{t_i}^k}^2\omega_{t_i}^n\le 1.
\end{equation}
Then, passing to a subsequence if necessary, $\varsigma_i$ converges to a locally bounded holomorphic section $\varsigma_\infty$ of $L_T^k$ over $\mathcal{R}$ which satisfies
\begin{equation}\label{gradient estimate: 2}
\sup_{B_{d_T}(x,r)\cap\mathcal{R}}|\nabla^{h_T}\varsigma_\infty|^2_{h_T^k\otimes\bar{\omega}_T}\le C(R)\cdot r^{-2n-2}\cdot k^{n+1}\int_{B_{d_T}(x,2r)\cap\mathcal{R}}|\varsigma_\infty|_{h_T^k}^2\bar{\omega}_T^n,
\end{equation}
whenever $B_{d_T}(x,r)\subset B_{d_T}(x_T,R)$.
\end{prop}

\subsection{Algebraic structure of $M_T$}

Recall the holomorphic map $\Phi:L\rightarrow\mathbb{C}P^N$ defined by an orthonormal basis of $(L,H_0)$ where $N=\dim H^0(M;L)-1$, $H_0$ is any fixed Hermitian metric. For any time $t$, the map
$$\Phi_t=\Phi:(M,\omega_t)\rightarrow(\Phi(M),\omega_{FS})$$
is Lipschitz with a uniform Lipschitz constant, since $\Phi_t^*\omega_{FS}=\ell_0\eta_T\le C\ell_0\omega_t$ for any $t\le T$. Let $t_i\rightarrow T$ be a sequence of times and $(M,\omega_{t_i},x_0)\stackrel{d_{GH}}{\longrightarrow}(M_T,d_T,x_T)$ be the Gromov-Hausdorff convergence considered in Subsection 3.2, then by a diagonalization argument, up to taking a subsequence, $\Phi_{t_i}$ converges to a Lipschitz map
$$\Phi_T=\lim_{t_i\rightarrow T}\Phi_{t_i}:(M_T,d_T)\rightarrow(\Phi(M),\omega_{FS}).$$

\begin{prop}\label{limit map: injectivity}
$\Phi_T$ is injective.
\end{prop}

If $M_T$ is compact, then $\Phi_T$ is a homeomorphism. In general we have

\begin{prop}\label{limit map: local homeomorphism}
$\Phi_T$ is a local homeomorphism.
\end{prop}

Therefore the limit $M_T$ is locally algebraic. The procedure to prove the propositions are contained in the work of Donaldson-Sun \cite{DoSu14} and Tian \cite{Ti13}, \cite{Ti12}. We outline a proof following \cite{Ti12}.

The Hermitian line bundles $(L,h_{t_i})$ have curvature $\Theta_{h_{t_i}}=\ell_0\frac{T}{t_i}\omega_{t_i}-\ell_0\frac{T-t_i}{t_i}\omega_0$. As $t_i\rightarrow T$, the Hermitian line bundles $(L,h_{t_i})$ converges smoothly to a limit Hermitian line bundle $(L_T,h_T)$ over $\mathcal{R}$ whose curvature
\begin{equation}
\Theta_{h_T}=\ell_0\bar{\omega}_T,\,\mbox{ on }\mathcal{R}.
\end{equation}

\vskip 0.1in
\noindent
Step 1: Construction of local approximating holomorphic sections.

Let $p\in M_T$ be any point. Let $r_j\rightarrow 0$ be a decreasing sequence of radius, $k_j=r_j^{-2}\in\mathbb{Z}$, and $\mathcal{C}_p=\lim_{j\rightarrow\infty}(M_T,r_j^{-1}d_T,p)$ be a tangent cone at $p$. By the regularity theory of Cheeger-Colding \cite{ChCo97, ChCo00} and Cheeger-Colding-Tian \cite{ChCoTi}, we have

\vskip 0.1in
\noindent
${\bf T}_1$. ${\cal C}_p$ is smooth outside a closed subcone ${\cal S}_p$ of complex codimension at least $1$ which is the singular set of ${\cal C}_p$;

\vskip 0.1in
\noindent
${\bf T}_2$. There is a K\"ahler Ricci-flat cone metric $\omega_p$ of the form $\sqrt{-1} \,\partial\bar\partial \rho^2$ on ${\cal C}_p \backslash {\cal S}_p$,
where $\rho$ denotes the distance function from the vertex of ${\cal C}_p$, denoted by $o$.

\vskip 0.1in
\noindent
${\bf T}_3$. Denote by $L_p$ the trivial bundle ${\cal C}_p\times \CC$ over ${\cal C}_p$ equipped with
the Hermitian metric $e^{-\ell_0\rho^2}\,|\cdot |^2$. The curvature of this Hermitian metric is given by $\omega_p$.

For any $ \epsilon \,>\,0$, we put
$$V(p; \epsilon)\,=\,\{ \,y \,\in \, {\cal C}_p \,|\, y\,\in\, B_{\epsilon^{-1}}(o,\omega_p)\,\backslash \,\overline{B_{\epsilon}(o,\omega_p)},\,\, d(y, {\cal S}_p )\, > \,\epsilon\,\,\}.$$
For any $\epsilon > 0$ and $\delta > 0$, we can have a $j_0=j_0(\epsilon, \delta)$ such that $r_{j_0}\le\epsilon^2$,
and for each $j \ge j_0$, there is a diffeomorphism $\phi_j: V(p;\frac{\epsilon}{4})\rightarrow  M_T\backslash {\cal S}$,
where ${\cal S}$ is the singular set of $M_T$,
satisfying:
\vskip 0.1in
\noindent
(i) $d(p, \phi_j(V(p; \epsilon))) \,<\, 10\, \epsilon r_j$ and $\phi_j(V(p;\epsilon)) \subset B_{(1+\epsilon^{-1}) r_j}(p)$;
\vskip 0.1in
\noindent
(ii) The K\"ahler metric $\omega_T$ on $M_T\backslash {\cal S}$ satisfies
\begin{equation}
\label{eq: bound-1}
||r_j^{-2} \phi_j^*\omega_T \,-\, \omega_p ||_{C^6(V(p; \frac{\epsilon}{2}))} \,\le\,\delta,
\end{equation}
where the norm is defined in terms of the metric $\omega_p$.

\begin{lemm}
\label{lemm:partial-5}
Given $\epsilon > 0 $ and any sufficiently small $\delta > 0$, there are a sufficiently large $j$, a diffeomorphism
$\phi_j: V(p;\frac{\epsilon}{4})\rightarrow  M_T\backslash {\cal S}$ with properties {\rm(i)} and {\rm(ii)} above,
and an isomorphism $\psi_j $ from the trivial bundle ${\cal C}_p\times \CC$ onto
$L^{k_j}$ over $V(p;\epsilon)$ commuting with $\phi_j$ satisfying:
\begin{equation}
\label{eq:est-2}
|\psi_j(1)|_{h_T}^{2} \,=\, e^{-\ell_0\rho^2} ~~~\mbox{ and }~~~ ||\nabla \psi_j ||_{C^6(V(p; \epsilon))} \,\le \, \delta,
\end{equation}
where $\nabla$ denotes the covariant derivative with respect to the metrics
$h_T$ and $e^{-\ell_0\rho^2}\, |\cdot |^2$.
\end{lemm}

We refer the readers to \cite[Lemma 5.7]{Ti12} for its proof. The proof uses that the limit line bundle $(L_T,h_T)$ is Hermitian Einstein, namely $\Theta_{h_T}=\ell_0\omega_T$. We also need the following lemma.

\begin{lemm}
\label{lemm:partial-6}
For any $\bar \epsilon \,>\,0$, there is a smooth function $\gamma_{\bar\epsilon}$ on ${\cal C}_p$ satisfying:

\vskip 0.1in
\noindent
{\rm(1)} $\gamma_{\bar\epsilon }(y)\,=\,1$ if $d(y,{\cal S}_p)\,\ge\,\bar\epsilon$;

\vskip 0.1in
\noindent
{\rm(2)} $0\,\le\,\gamma_{\bar\epsilon} \,\le\,1$ and $\gamma_{\bar\epsilon} (y)\,=\,0$ in an neighborhood of ${\cal S}_p$;

\vskip 0.1in
\noindent
{\rm(3)} $|\nabla \gamma_{\bar\epsilon}|\,\le\, C$ for some constant $C\,=\,C(\bar\epsilon)$ and
$$\int_{B_{{\bar\epsilon}^{-1}}(o,\omega_p)} \,|\nabla \gamma_{\bar\epsilon}|^2\, \omega_p^n\,\le\,\bar\epsilon.$$
\end{lemm}

This is exactly Lemma 5.8 of \cite{Ti12}. It holds trivially for two simple cases: (a) ${\cal S}_p$ is of codimension at least $4$ and (b) the tangent cone ${\cal C}_p$ splits as $\CC^{n-1}\times {\cal C}_p'$ where ${\cal C}_p'$ is a 2-dimensional flat cone. The general case when ${\cal S}_p$ is of codimension $2$ can be proved by recursion based on the construction of peak holomorphic section in Step 2 below; see Appendix A of \cite{Ti12}. If $x\in{\cal C}_p$ has an iterated tangent cone of form $\CC^{n-1}\times{\cal C}_x'$, then the singular set around $x$ is a locally analytical set modeled by taking the limit of the ample locus $M_{\amp}$ on the original manifold.

Assuming the two lemmas, one can find for any $\bar{\epsilon}$ one $0<\epsilon=\epsilon(\bar{\epsilon})<\bar{\epsilon}$ such that $\supp(\gamma_{\bar{\epsilon}})\subset V(p;\epsilon)$ and then one $j=j(\bar{\epsilon})$ satisfying Lemma \ref{lemm:partial-5} for $\epsilon$. Then $\tau=\psi_j(\gamma_{\bar{\epsilon}} 1)$ extends to a smooth section of $L_T^{k_j}$ on $M_T$ which satisfies: (4) $\tau$ is holomorphic on $\phi_j(V(p;\bar{\epsilon}))$ and (5)
$$\int_{M_T}|\bar{\partial}\tau|_{h_T^{k_j}\otimes k_j\omega_T}^2(k_j\omega_T)^n\le\bar{\epsilon}.$$

\vskip 0.1in
\noindent
Step 2: Existence of holomorphic peak sections on $M$.

Let $p\in M_T$ satisfies $d_T(p,x_T)\le R$. Suppose $p_i\in M$ satisfies $p_i\stackrel{d_{GH}}{\longrightarrow}p$ under the Cheeger-Gromov convergence. By the smooth convergence on the regular set $\phi_j(V(p;\epsilon))$, the approximating holomorphic section $\tau$ on $L_T^{k_j}$ descends to a family of smooth section of $L^{k_j}$ on $M$, denoted $\tau_i$, via a family of smooth maps $f_i:\mathcal{R}\rightarrow M$ representing the Gromov-Hausdorff convergence $(M,\omega_{t_i},x_0)\stackrel{d_{GH}}{\longrightarrow}(M_T,d_T,x_T)$, which satisfies:
\begin{equation}\label{separating: e 0}
\supp(\tau_i)\subset f_i\big(\phi_j(V(p;\epsilon))\big)\subset B_{k_j\omega_{t_i}}(p_i,2\sqrt{k_j}\bar{\epsilon}),
\end{equation}
\begin{equation}\label{separating: e 1}
\bigg||\tau_i|_{h_{t_i}^{k_j}}(x)-e^{-\ell_0d^2_{k_j\omega_{t_i}}(x,p_i)}\bigg|\le\bar{\epsilon},\,\mbox{ on }f_i\big(\phi_j(V(p;\bar{\epsilon}))\big),
\end{equation}
and
\begin{equation}\label{separating: e 2}
C^{-1}\le\int_{M}|\tau_i|_{h_{t_i}^{k_j}\otimes k_j\omega_{t_i}}^2(k_j\omega_{t_i})^n\le C,
\end{equation}
for some constant $C=C(R)$ depending on the volume ration of the tangent cone $\mathcal{C}_p$, and
\begin{equation}\label{separating: e 3}
\int_{M}|\bar{\partial}\tau_i|_{h_{t_i}^{k_j}\otimes k_j\omega_{t_i}}^2(k_j\omega_{t_i})^n\le2\bar{\epsilon},
\end{equation}
for any $i$ sufficiently large. We may assume that
\begin{equation}\label{separating: e 4}
f_i\big(\phi_j(V(p;2\bar{\epsilon}^{\frac{1}{4n}}))\big)\cap B_{k_j\omega_{t_i}}(p_i,4\bar{\epsilon}^{\frac{1}{4n}})\neq\emptyset,
\end{equation}
\begin{equation}\label{separating: e 5}
B_{k_j\omega_{t_i}}(p'_i,\frac{1}{2}\bar{\epsilon}^{\frac{1}{4n}})\subset f_i\big(\phi_j(V(p;\bar{\epsilon}))\big),
\end{equation}
for some $p'_i\in M$ with $d_{k_j\omega_{t_i}}(p_i,q_i)\le2\bar{\epsilon}^{\frac{1}{4n}}$. When $i$ is large enough, $T-t_i\le\frac{1}{2\ell_0k_j}$, so by the $L^2$ estimate in Lemma \ref{L^2 estimate: prop 1}, there exists a smooth section $v_i$ solving $\bar{\partial}v_i=\bar{\partial}\tau_i$ with
\begin{equation}\label{separating: e 6}
\int_M|v_i|^2_{h_{t_i}^{k_j}}(k_j\omega_{t_i})^n\le C\cdot\bar{\epsilon}
\end{equation}
for some $C$ independent of $i$. Noticing that $v_i$ is holomorphic on $B_{k_j\omega_{t_i}}(p'_i,\frac{1}{2}\bar{\epsilon}^{\frac{1}{4n}})$, by the $L^\infty$ estimate in Lemma \ref{infinity estimate},
\begin{equation}\label{separating: e 7}
|v_i|^2_{h_{t_i}^{k_j}}(p'_i)\le C\bar{\epsilon}^{-\frac{1}{2}}\int_{B_{k_j\omega_{t_i}}(q_i,\frac{1}{2}\bar{\epsilon}^{\frac{1}{4n}})}
|v_i|^2_{h_{t_i}^{k_j}}(k_j\omega_{t_i})^n\le C\bar{\epsilon}^{\frac{1}{2}},
\end{equation}
where $C=C(R)$. Therefore, $\sigma_i=\tau_i-v_i$ defines a holomorphic section of $L^k$ satisfying
\begin{equation}\label{separating: e 8}
|\sigma_i|_{h_{t_i}^{k_j}}(p'_i)\ge e^{-\ell_0d^2_{k_j\omega_{t_i}}(p_i,p'_i)}-\bar{\epsilon}-C\bar{\epsilon}^{\frac{1}{4}}\ge\frac{1}{2}
\end{equation}
once $\bar{\epsilon}=\epsilon(p)$ is chosen sufficiently small. On the other hand, $\sigma_i=v_i$ outside  $f_i\big(\phi_j(V(p;\epsilon))\big)$, a domain satisfying
\begin{equation}\nonumber
\sup_{x\in f_i\big(\phi_j(V(p;\epsilon))\big)}d_{k_j\omega_{t_i}}(p_i,x)\le\sup_{x\in \phi_j(V(p;\epsilon))}d_{k_j d_T}(p,x)+1\le\epsilon^{-1}+2.
\end{equation}
Therefore, for any $x\in M$ with $\epsilon^{-1}+3 \le d_{k_j\omega_{t_i}}(x,p_i)\le 2k_j^{\frac{1}{2}}R$, an iteration gives
$$|\sigma_i|_{h_{t_i}^{k_j}}^2(x)\le C\int_{B_{k_j\omega_{t_i}}(x,1)}|\sigma_i|^2_{h_{t_i}^{k_j}}(k_j\omega_{t_i})^n\le C\cdot\bar{\epsilon}$$
where $C=C(R)$. Noticing that $k_j=r_j^{-2}$ and $r_j\le\epsilon^2$, we conclude
\begin{equation}\label{separating: e 10}
|\sigma_i|_{h_{t_i}^{k_i}}\le C\bar{\epsilon}^{\frac{1}{2}},\,\mbox{ on }B_{\omega_{t_i}}(x,2R)\backslash B_{\omega_{t_i}}(p_i,2\epsilon).
\end{equation}
Besides (\ref{separating: e 8}) and (\ref{separating: e 10}) the section $\sigma_i$ also satisfies
\begin{equation}\label{separating: e 11}
C^{-1}\le\int_{B_{\omega_{t_i}}(p_i,2\bar{\epsilon})}|\sigma_i|^2_{h_{t_i}^{k_j}}(k_j\omega_{t_i})^n\le C
\end{equation}
and
\begin{equation}\label{separating: e 12}
\int_{M\backslash B_{\omega_{t_i}}(p_i,2\bar{\epsilon})}|\sigma_i|^2_{h_{t_i}^{k_j}}(k_j\omega_{t_i})^n\le C\cdot\bar{\epsilon}
\end{equation}
for some $C=C(R)$ independent of the specified $i$, $j$ and $\varsigma$.

Passing to a subsequence if necessary, the sequence of points $p'_i$ converge to a point $p'\in\mathcal{R}$ with $d_T(p,p')\le2k_j^{-\frac{1}{2}}\bar{\epsilon}^{\frac{1}{4n}}$, the sections $\sigma_i\in H^0(M;L^k)$ converges to a holomorphic section $\sigma_\infty\in H^0(\mathcal{R};L_T^k)$ such that
\begin{equation}\label{separating: e 21}
|\sigma_\infty|_{h_T^{k_j}}(p')\ge \frac{1}{2},
\end{equation}
\begin{equation}\label{separating: e 22}
|\sigma_\infty|_{h_T^{k_i}}\le C\cdot\bar{\epsilon}^{\frac{1}{2}},\,\mbox{ on }B_{d_T}(x_T,2R)\backslash B_{d_T}(p,2\epsilon),
\end{equation}
\begin{equation}\label{separating: e 23}
C^{-1}\le\int_{B_{d_T}(p,2\bar{\epsilon})\cap\mathcal{R}}|\sigma_\infty|_{h_T^{k_j}}^2(k_j\bar{\omega}_T)^n\le C
\end{equation}
and
\begin{equation}\label{separating: e 24}
\int_{\mathcal{R}\backslash B_{d_T}(p,2\bar{\epsilon})}|\sigma_\infty|_{h_T^{k_j}}^2(k_j\bar{\omega}_T)^n\le C\cdot\bar{\epsilon}
\end{equation}
for some $C=C(R)$. By the gradient estimate in Proposition \ref{gradient estimate: prop 2} we have
\begin{equation}\label{separating: e 25}
|\sigma_\infty|_{h_T^{k_j}}(p)\ge \frac{1}{2}-C\bar{\epsilon}^{\frac{1}{4n}}\ge\frac{1}{4}
\end{equation}
if $\bar{\epsilon}$ is chosen sufficiently small.

\vskip 0.1in
\noindent
Step 3: $\Phi_T$ is injective.

For any $R>0$, $p,q\in B_{d_T}(x_T,R)$, and any $\bar{\epsilon}<< d_T(p,q)$, there is an integer $k=k(p,q)$ and sections $\sigma_{p}\in H^0(\mathcal{R};L_T^{k})$, $\sigma_{q}\in H^0(\mathcal{R};L_T^{k})$ such that (\ref{separating: e 22})-(\ref{separating: e 25}) hold respectively at $p_i,q_i$, for $C=C(R)$. Notice that by Schwarz inequality,
\begin{align}\nonumber
\bigg|\int_{\mathcal{R}}\langle\sigma_{p},\bar{\sigma}_{q}\rangle_{h_{T}^k}\bar{\omega}_{T}^n\bigg|
&\le\int_{\mathcal{R}\backslash B_{d_T}(q,2\bar{\epsilon})}|\sigma_{p}||\sigma_{q}|\bar{\omega}_{T}^n
+\int_{\mathcal{R}\backslash B_{d_T}(p,2\bar{\epsilon})}|\sigma_{p}||\sigma_{q}|\bar{\omega}_{T}^n\\
&\le Ck^{-\frac{n}{2}}\bigg(\big(\int_{\mathcal{R}\backslash B_{d_T}(q,2\bar{\epsilon})}|\sigma_{q}|^2\bar{\omega}_{T}^n\big)^{\frac{1}{2}}
+\big(\int_{\mathcal{R}\backslash B_{d_T}(p,2\bar{\epsilon})}|\sigma_{p}|^2\bar{\omega}_{T}^n\big)^{\frac{1}{2}}\bigg)
\nonumber\\
&\le Ck^{-n}\bar{\epsilon}^{\frac{1}{2}}.\nonumber
\end{align}
Let $\tilde{\sigma}_{p}$ and $\tilde{\sigma}_{q}$ respectively be the unit normalization of $\sigma_{p}$ and $\sigma_{q}-\langle \sigma_{p},\bar{\sigma}_{q}\rangle_{h_{t}^k}\sigma_{q}$, then $\tilde{\sigma}_{p}$ is orthogonal to $\tilde{\sigma}_{q}$, and
$$|\tilde{\sigma}_{p}|_{h_{T}^k}(p)\ge \frac{1}{4},\, |\tilde{\sigma}_{p}|_{h_{T}^k}(q)\le C\bar{\epsilon}^{\frac{1}{2}},$$
$$|\tilde{\sigma}_{q}|_{h_{T}^k}(p)\le C\bar{\epsilon}^{\frac{1}{2}},\, |\tilde{\sigma}_{q}|_{h_{T}^k}(q)\ge \frac{1}{4}-C\bar{\epsilon}^{\frac{1}{2}},$$
where $C=C(R)$. Thus,
$$\bigg|\frac{\tilde{\sigma}_{p}(p)}{\tilde{\sigma}_{q}(p)}\bigg|
\ge\frac{1}{4C\bar{\epsilon}^{\frac{1}{2}}},\mbox{ and, }
\bigg|\frac{\tilde{\sigma}_{p}(q)}{\tilde{\sigma}_{q}(q)}\bigg|
\le\frac{C\bar{\epsilon}^{\frac{1}{2}}}{8}.$$
Denote by $\Phi_{t}^k:M\rightarrow\mathbb{C}P^{N_k}$, $N_k=\dim H^0(M;L^k)-1$, the holomorphic map defined by an orthonormal basis of $(L,h_t^k)$. Let $p_i\rightarrow p$, $q_i\rightarrow q$ under the Gromov-Hausdorff convergence.  Then
$$\lim_{t_i\rightarrow T}d_{FS}(\Phi_{t_i}^k(p_i),\Phi_{t_i}^k(q_i))\ge C_5^{-1}$$
for some $C_5=C_5(R)>0$. By the uniform equivalence of $h_t$, we have
$$\lim_{t_i\rightarrow T}d_{FS}(\Phi_{0}^k(p_i),\Phi_{0}^k(q_i))\ge C_6^{-1}$$
for some $C_6=C_6(R,k)>0$. Then applying the effective version of the finite generation of the canonical ring $\bigoplus_{k\ge 0}H^0(M;L^k)$, see the discussion in \S 3.1, we have
$$\lim_{t_i\rightarrow T}d_{FS}(\Phi(p_i),\Phi(q_i))\ge C_7^{-1}$$
for some $C_7=C_7(R,k)=C_7(R,p,q)>0$. It means that
$$d_{FS}(\Phi_T(p),\Phi_T(q))\ge C_7^{-1}.$$
Therefore, the map $\Phi_T$ is an injection.

\vskip 0.1in
\noindent
Step 4: $\Phi_T$ is a local homeomorphism.

By the discussion in Step 3, together with the relative $C^0$ estimate in Proposition \ref{gradient estimate: prop 1}, for any $p\in B(x_T,R)$ and $q$ with $d_T(p,q)=1$, there exists $\delta(p,q)>0$ and $r(p,q)>0$ such that
$$d_{FS}(\Phi(x),\Phi(p))\ge\delta(p,q),\,\forall x\in B(q,r(p,q)).$$
Since $\partial B(p,1)$ is compact, we can find a uniform $\delta>0$ such that
\begin{equation}
d_{FS}(\Phi(q),\Phi(p))\ge\delta,\,\forall q\in\partial B(p,1).
\end{equation}
It follows that $\Phi$ is an open map. Since $\Phi$ is also an injection, it must be a local homeomorphism.



\section{Metric structure of the limit space}

Let $(M,\omega_t)$ be the solution to the continuity equation (\ref{M-A: continuity}) and $L$ be the limit line bundle defined in the previous section.

\subsection{Diameter bound of the singular K\"ahler metric}

In \cite{So14}, Song developed a method to prove the diameter bound of a singular K\"ahler-Einstein metric. In this subsection, we follow his idea (see also \cite{Gu15}) to show the diameter bound of $(M\backslash D,\omega_T)$ where $D$ is any divisor such that $[\omega_0]-Tc_1(M)-\epsilon[D]>0$ for some $\epsilon>0$. A bit difference is that the metric $\omega_T$ is a twisted K\"ahler-Einstein metric.

More precisely, let $p\in D$ be any point, $\pi:\widetilde{M}\rightarrow M$
be the blow-up at $p$ with exceptional divisor $\pi^{-1}(p)=E$. Then
$$K_{\widetilde{M}}=\pi^*K_M+(n-1)E.$$
Let $h_E$ be the Hermitian metric on $L_E$, the associated line bundle of $E$, and $\sigma_E$ be a defining section. We denote by $\widetilde{D}=\sum a_i\widetilde{D}_i$ where $\widetilde{D}_i=\pi^{-1}D_i$ is the total transformation and $h_{\widetilde{D}}=\pi^*h_D$, Hermitian metrics on $L_{\widetilde{D}}$. Let $\chi$ be a fixed K\"ahler metric on $\widetilde{M}$. By the calculation in \cite[Page 186]{GrHa}, see also \cite{SoWe13}, the metric $h_E$ can be chosen such that
\begin{equation}\nonumber
\pi^*\eta_T+\delta_0\sqrt{-1}\partial\bar{\partial}\log\|\sigma_E\|_{h_E}^2\ge\delta_1\chi
\end{equation}
for some small constants $\delta_0,\delta_1>0$. Observe that
\begin{equation}\nonumber
\widetilde{\Omega}=\|\sigma_E\|_{h_E}^{-2(n-1)}\pi^*\Omega
\end{equation}
defines a smooth volume form on $\widetilde{M}$. Consider the following family of Monge-Amp\`ere equations on $\widetilde{M}$, for $0<\epsilon\le\frac{1}{2}\delta_0$,
\begin{equation}\label{M-A: epsilon}
\big(\pi^*\eta_T+\epsilon\chi
+\sqrt{-1}\partial\bar{\partial}\tilde{\varphi}_\epsilon\big)^n
=e^{\frac{1}{T}\tilde{\varphi}_\epsilon}\big(\|\sigma_E\|_{h_E}^{2}+\epsilon^2\big)^{n-1}
\widetilde{\Omega}.
\end{equation}
By Yau's solution to Calabi problem \cite{Ya78}, the equation has a unique smooth solution $\tilde{\varphi}_\epsilon$,  for all $0<\epsilon\le\frac{1}{2}\delta_0$; moreover,
\begin{equation}
\tilde{\omega}_\epsilon=:\pi^*\eta_T+\epsilon\chi
+\sqrt{-1}\partial\bar{\partial}\tilde{\varphi}_\epsilon
\end{equation}
is a smooth K\"ahler metric on $\widetilde{M}$.

\begin{lemm}
There exists $C$ independent of $\epsilon$ such that
\begin{equation}\label{C^0 estimate: 1}
-C\le\tilde{\varphi}_\epsilon\le C-(n-1)T\log\big(\|\sigma_E\|_{h_E}^2+\epsilon^2\big).
\end{equation}
\end{lemm}
\begin{proof}
The proof uses simply the maximum principle. To get the upper bound, we rewrite the Monge-Amp\`ere equation (\ref{M-A: epsilon}) as
\begin{eqnarray}\nonumber
\big(\theta+\sqrt{-1}\partial\bar{\partial}(\tilde{\varphi}_\epsilon+f_\epsilon)\big)^n=e^{\frac{1}{T}
(\tilde{\varphi}_\epsilon+f_\epsilon)}\widetilde{\Omega},\,\mbox{ on }\widetilde{M}
\end{eqnarray}
where $f_\epsilon=(n-1)T\log\big(\|\sigma_E\|_{h_E}^2+\epsilon^2\big)$ and $\theta=\pi^*\eta_T+\epsilon\chi
-\sqrt{-1}\partial\bar{\partial}f_\epsilon$. By direct calculation, cf. \cite[Lemma 4.9]{So14},
\begin{equation}\nonumber
\sqrt{-1}\partial\bar{\partial}\log\big(\|\sigma_E\|_{h_E}^{2}+\epsilon^2\big)^{n-1}\ge-C\chi.
\end{equation}
Therefore, $\theta\le C\chi$. By maximum principle we get $\sup(\tilde{\varphi}_\epsilon+f_\epsilon)\le C$. To get the lower bound we recall that
$$C^{-1}\|\sigma_E\|_{h_E}^{2(n-1)}\chi^n\le\big(\pi^*\eta_T\big)^n\le C\|\sigma_E\|_{h_E}^{2(n-1)}\chi^n$$
for some $C>0$ depending on the choice of $\sigma_E$, $h_E$, $\hat{\eta}_T$ and $\chi$. By maximum principle,
$$e^{\frac{1}{T}\inf\tilde{\varphi}_\epsilon}\ge\inf\frac{\big(\pi^*\eta_T+\epsilon\chi\big)^n}
{\big(\|\sigma_E\|_{h_E}^{2}+\epsilon^2\big)^{n-1}\widetilde{\Omega}}
\ge C^{-1}\inf\frac{\|\sigma_E\|_{h_E}^{2(n-1)}+\epsilon^n}{\big(\|\sigma_E\|_{h_E}^{2}+\epsilon^2\big)^{n-1}}.$$
The right hand side has a uniform lower bound when $0<\epsilon\le \frac{1}{2}\delta_0$.
\end{proof}

\begin{lemm}
There exist $\lambda$ and $C$ independent of $\epsilon$ such that
\begin{equation}
\|\tilde{\varphi}_\epsilon\|_{C^0}\le C
\end{equation}
and
\begin{equation}
\tilde{\omega}_\epsilon\le\frac{C}{\|\sigma_E\|_{h_E}^{2\lambda}\|\tilde{\sigma}\|^{2\lambda}}\chi.
\end{equation}
Moreover, for any compact subset $K\subset\widetilde{M}\backslash\widetilde{D}$ and integer $k$ there exists $C_{K,k}$ independent of $\epsilon$ such that
\begin{equation}
\|\tilde{\omega}_\epsilon\|_{C^k(K)}\le C_{K,k}.
\end{equation}
\end{lemm}
\begin{proof}
The upper bound of $\tilde{\varphi}_\epsilon$ implies that $\tilde{\omega}^n_\epsilon\le C\chi^n$ for some $C$ independent of $\epsilon$. The $C^0$ estimate of $\tilde{\varphi}_\epsilon$ follows from \cite{Zh06} or \cite{EyGuZe09}. The proof of the next two estimates are standard; see \cite{So14} for details.
\end{proof}

Furthermore, as in \cite{So14}, the estimate of $\tilde{\omega}_\epsilon$ can be improved in the "normal direction" along the proper transformation of $D$, say $F$, which is the closure of $\pi^{-1}(D\backslash\{p\})$. Locally, let $B$ be a disk centered at $p$ and $\widetilde{B}=\pi^{-1}(B)$. Let $f_1,\cdots,f_N$ be the defining functions of the divisor $F$ in $\widetilde{B}$. Then, totally as in \cite{So14}, we can prove

\begin{prop}
There exists $\delta>0$, $\lambda>0$ and $C>0$ independent of $\epsilon$ such that
\begin{equation}
\tilde{\omega}_\epsilon\le\frac{C}{\|\sigma_E\|_{h_E}^{2(1-\delta)}\prod|f_i|^{2\lambda}}\chi,\,\mbox{ in }\widetilde{B}.
\end{equation}
\end{prop}

\begin{prop}
$\tilde{\omega}_{\epsilon}$ converges to $\pi^*\omega_T$ as $\epsilon\rightarrow 0$ in the current sense. Moreover, the convergence takes place smoothly on $\widetilde{M}\backslash\widetilde{D}$.
\end{prop}
\begin{proof}
It suffices to show that for any sequence $\tilde{\omega}_i=\tilde{\omega}_{\epsilon_i}$ with $\epsilon_i\rightarrow 0$, if $\tilde{\omega}_i\rightarrow\tilde{\omega}_0$ in the current sense, then $\tilde{\omega}_0=\pi^*\omega_T$.

Write $\tilde{\omega}_i=\pi^*\eta_T+\epsilon_i\chi+\sqrt{-1}\partial\bar{\partial}\tilde{\varphi}_i$, then $\tilde{\varphi}_i\rightarrow\tilde{\varphi}_0$ for some $\pi^*\eta_T$-plurisubharmonic function $\tilde{\varphi}_0$ such that $\tilde{\omega}_0=\pi^*\eta_T+\sqrt{-1}\partial\bar{\partial}\tilde{\varphi}_0$. It is trivial that
$$\big(\pi^*\eta_T+\sqrt{-1}\partial\bar{\partial}\tilde{\varphi}_0\big)^n=e^{\frac{\tilde{\varphi}_0}{T}}
\pi^*\Omega.$$
Observe that $\pi^*u_T$ satisfies the same Monge-Amp\`ere equation
$$\big(\pi^*\eta_T+\sqrt{-1}\partial\bar{\partial}\pi^*u_T\big)^n=e^{\frac{\pi^*u_T}{T}}
\pi^*\Omega.$$
Furthermore, both $\tilde{\varphi}_0$ and $\pi^*u_T$ are uniformly bounded, both the volume forms $e^{\frac{\tilde{\varphi}_0}{T}}\pi^*\Omega$ and $e^{\frac{\pi^*u_T}{T}}\pi^*\Omega$ have full measure in the big cohomological class $[\omega_0]+TK_M$. One can use the comparison principle \cite[Corollary 2.3]{BEGZ} to conclude that $\tilde{\varphi}_0=\pi^*u_T$. In other words, $\tilde{\omega}_0=\pi^*\omega_T$.
\end{proof}

\begin{coro}\label{diameter bound: normal estimate}
Assume as above. There exist $\delta>0$, $\lambda>0$ and $C>0$ such that
\begin{equation}
\pi^*\omega_T\le\frac{C}{\|\sigma_E\|_{h_E}^{2(1-\delta)}\prod|f_i|^{2\lambda}}\chi,\,\mbox{ in }\widetilde{B}.
\end{equation}
\end{coro}

This implies that any point of $D$ is a finite point in the metric completion of $(M\backslash D,\omega_T)$. However, it can not be concluded from this the diameter bound of $(M\backslash D,\omega_T)$ .

From now on we turn to the Gromov-Hausdorff convergence. Let $t_i\rightarrow T$ be a sequence of times and $x_0\in M\backslash\mathcal{S}_M$ such that
\begin{equation}
(M,\omega_{t_i},x_0)\stackrel{d_{GH}}{\longrightarrow}(M_T,d_T,x_T).
\end{equation}
Let $\mathcal{R}$ be the regular set of $M_T$ with a $C^{1,\alpha}$ metric $\bar{\omega}_T$ for any $\alpha>0$ such that $\omega_{t_i}\stackrel{C^{1,\alpha}}{\longrightarrow}\bar{\omega}_T$. Let $\Phi:M\rightarrow \mathbb{C}P^{N}$ be the holomorphic map via an orthonormal basis of $(L,h_0)$ defined in the previous section, where $N=\dim H^0(M;L)-1$. After choosing a subsequence, the map $\Phi_{t_i}=\Phi:(M,\omega_{t_i})\rightarrow(\Phi(M),\omega_{FS})$ converges to a Lipschitz map
$$\Phi_T:(M_T,d_T)\rightarrow(\Phi(M),\omega_{FS})$$
by putting  $\Phi_T(x)=\lim\Phi_{t_i}(x_i)$ where $x_i\in M$ is an sequence satisfying $x_i\stackrel{d_{GH}}{\longrightarrow}x$. Since $M\backslash D\subset M\backslash\mathcal{S}_M\subset M_{\reg}$, the metric $\omega_{t_i}$ converges smoothly to $\omega_T$ on $M\backslash D$. Recall
$$D_T=\{x\in M_T|\exists \,x_i\in D\mbox{ such that }x_i\stackrel{d_{GH}}{\longrightarrow}x.\}.$$
The singular set $\mathcal{S}\subset D_T$; moreover, $(M\backslash D,\omega_T)$ is isometric to $(M_T\backslash D_T,\bar{\omega}_T)$.

\begin{lemm}
$\Phi_T:M_T\backslash D_T\rightarrow \Phi(M\backslash D)$ is a bijection.
\end{lemm}
\begin{proof}
For any $x\in M_T\backslash D_T$, there exists $x'\in M\backslash D$ such that $x'\stackrel{d_{GH}}{\longrightarrow}x$. Then $\Phi_T(x)=\lim\Phi_{t_i}(x')=\Phi(x')\in\Phi(M\backslash D)$, so $\Phi_T(M_T\backslash D_T)\subset \Phi(M\backslash D)$. Conversely, for any $x'\in M\backslash D$, we have $x'\stackrel{d_{GH}}{\longrightarrow}x$ for some $x\in M_T\backslash D_T$. This is because $d_{\omega_{t_i}}(x',D)\ge\delta$ uniformly for some $\delta>0$ independent of $i$. Thus, $\Phi(x')=\Phi_T(x)$, the map $\Phi_T$ is surjective onto $\Phi(M\backslash D)$.
\end{proof}

\begin{lemm}
$\Phi_T:D_T\rightarrow\Phi(D)$ is surjective.
\end{lemm}
\begin{proof}
Noticing that $\Phi(D)$ is compact, the limits of points in $\Phi(D)$ remains in this set, so $\Phi_T(D_T)\subset\Phi(D)$. On the other hand, by Corollary \ref{diameter bound: normal estimate}, for any $x'\in D$ there exists a curve $\gamma:[0,1]\rightarrow M$ with $\gamma(0)=x'$ and $\gamma\big((0,1]\big)\subset M\backslash D$ such that $$L_{\omega_T}(\gamma)=\int_0^1|\dot{\gamma}|_{\omega_T}dt<\infty.$$
Through an isometry from $(M\backslash D,\omega_T)$ to $(M_T\backslash D_T,\bar{\omega}_T)$, the curve $\gamma(t)$ gives a curve $\bar{\gamma}(t)$, $0<t\le 1$, which is bounded. Hence, there is a limit $x''=\lim_{t\rightarrow 0}\bar{\gamma}(t)$ in $M_T$. Then, by the continuity of $\Phi_T$,
$$\Phi_T(x'')=\lim_{t\rightarrow 0}\Phi_T(\bar{\gamma}(t))=\lim_{t\rightarrow 0}\lim_{i\rightarrow\infty}\Phi_{t_i}(\gamma(t))=\lim_{t\rightarrow 0}\Phi(\gamma(t))=\Phi(x').$$
Finally from above lemma we know that $x''\in D_T$, i.e., $\Phi(x')\in\Phi_T(D_T)$.
\end{proof}

\begin{coro}\label{limit map: surjectivity}
$\Phi_T$ is surjective.
\end{coro}

Combining with the Propositions \ref{limit map: injectivity} and \ref{limit map: local homeomorphism} we conclude that

\begin{prop}\label{limit map: homeomorphism}
$\Phi_T:M_T\rightarrow\Phi(M)$ is a homeomorphism. As a consequence, the diameter of $M_T$ is finite.
\end{prop}
\begin{proof}
It is clear that $\Phi(M)$ is compact, so $M_T$ is compact and has finite diameter.
\end{proof}

\begin{rema}
In {\rm\cite{Gu15}}, Guo presented another proof of the proposition without use of the local homeomorphic property.
\end{rema}

\begin{rema}
The continuity method provides an approach to construct K\"ahler currents in the big cohomological classes; as has shown, the K\"ahler currents are smooth on the ample locus. In general it is hard to detect the metric property of the K\"ahler currents.
\end{rema}

\subsection{Cheeger-Gromov convergence: diameter finiteness}

\begin{lemm}
There exists $C$ such that
\begin{equation}
\diam(M,\omega_t)\le C,\,\forall 0<t<T.
\end{equation}
\end{lemm}
\begin{proof}
It suffices to prove a uniform diameter bound of $\omega_t$ when $t$ is close to $T$. This is simply a consequence of the relative volume comparison theorem.

Recall that by the formula (\ref{M-A: continuity}) we may assume that $\Ric(\omega_t)$ is bounded below uniformly, say
$$\Ric(\omega_t)\ge-(n-1)\Lambda\omega_t,\,\forall \frac{T}{2}\le t\le T.$$
Denote by $R_0=\diam_{d_T}(M_T)$. Let $\varepsilon>0$ be any number and $D$ be any divisor such that $\mathcal{S}_M\subset D$. Since the regular set $\mathcal{R}$ is geodesically convex, we can choose $K\subset M\backslash D$, a connected and compact subset such that $\vol_{\omega_T}(M\backslash K)\le \varepsilon$ and $\diam_{d_{\omega_T}}(K)\le 2R_0$. By smooth convergence on $K$, we have
$$\vol_{\omega_t}(M\backslash K)\le 2\varepsilon$$
and
$$\diam_{d_{\omega_t}}(K)\le 2\diam_{d_{\omega_T}}(K)\le 2R_0$$
for $t$ sufficiently close to $T$, where the diameter is measured with respect to the intrinsic length metric induced by $\omega_t$ on $K$.

Suppose $x_t\in M\backslash K$ achieves maximum distance to $K$ in $(M,\omega_t)$ and put $R_1=R_1(t)=d_{\omega_t}(x_t,K)$. Then by the relative volume comparison theorem we have, when $t$ is chosen close to $T$,
\begin{eqnarray}
\frac{\vol_{\omega_t}(M)}{\varepsilon}\le\frac{\vol_{\omega_t}(B_{2R_0+R_1}(x_t))}{\vol_{\omega_t}(B_{R_1}(x_t))}
\le\frac{\int_0^{2R_0+R_1}\sinh^{n-1}\big(\sqrt{\Lambda}t\big)dt}
{\int_0^{R_1}\sinh^{n-1}\big(\sqrt{\Lambda}t\big)dt}.\nonumber
\end{eqnarray}
If $\varepsilon=\varepsilon(R_0,\Lambda,n,\vol_{\omega_t}(M))$ is chosen sufficiently small, this leads to a desired upper bound of $R_1$ in terms of $\vol_{\omega_t}(M)$, $\varepsilon$, $n$, $\Lambda$ and $R_0$, which does not depend on the time $t$.
\end{proof}

Now, summing up the discussions in \S 3.2, \S 3.7 and this subsection, we end the proof of Theorem \ref{convergence: CG}.



\section{Some examples}


\subsection{Minimal models of general type}

In the case when $M$ is a smooth minimal model of general type, the continuity equation (\ref{M-A: continuity}) is solvable for all $t>0$. We normalize the equation as follows
\begin{equation}\label{M-A: continuity-normalization}
(1+t)\omega=\omega_0-t\Ric.
\end{equation}
$\omega(t)$ is a family of solution to this equation iff $\frac{1}{1+t}\omega(\frac{t}{1+t})$ solves the initial equation (\ref{M-A: continuity}). Therefore, \ref{M-A: continuity-normalization} is solvable for all $t>0$. Moreover, as in the case of K\"ahler-Ricci flow, cf. \cite{TiZh06, So14}, we have

\begin{theo}\label{theo: minimal model}
When $M$ is a smooth minimal model of general type, the solution $\omega(t)$ of {\rm(\ref{M-A: continuity-normalization})} satisfies
\begin{itemize}
\item[{\rm(1)}] $\omega(t)$ converges as $t\rightarrow\infty$ in the current sense to a positive current $\omega_\infty\in-2\pi c_1(M)$ satisfying $\Ric(\omega_\infty)=-\omega_\infty$,

\item[{\rm(2)}] $\omega(t)$ converges smoothly to $\omega_\infty$ outside the exceptional locus $M_{\exc}$ of the birational morphism to the canonical model of $M$,

\item[{\rm(3)}] the metric completion of $(M\backslash M_{\exc},\omega_\infty)$ is a compact length metric space, denoted $(M_\infty,d_\infty)$, which is homeomorphic to the canonical model of $M$,

\item[{\rm(4)}] $(M,\omega(t))$ converges in the Cheeger-Gromov sense to $(M_\infty,d_\infty)$.
\end{itemize}
\end{theo}

In K\"ahler-Ricci flow, the global Cheeger-Gromov convergence of $\omega(t)$ is still unknown; see Conjecture 4.1 in \cite{So14}. When the Ricci curvature is bounded below, the conjecture is confirmed by Guo \cite{Gu15}.

\subsection{Algebraic contractions}

Let $X$ be a normal projective variety and $\pi:M\rightarrow X$ be an algebraic contraction with exceptional divisor $E$. $E$ has complex codimension 1 or at least 2 respectively when $\pi$ is a divisoral contraction or a small contraction. According to the discussion in the proof of Theorem \ref{convergence: CG}, cf. Sections 3 and 4, we have

\begin{theo}\label{theo: contraction}
Assume as above. Let $\omega_0$ be an initial K\"ahler metric in $2\pi c_1(L')$ for some line bundle $L'$. Then the solution $\omega(t)$ to {\rm(\ref{M-A: continuity})} satisfies
\begin{itemize}
\item[{\rm(1)}] $\omega(t)$ converges as $t\rightarrow T$ in the current sense to a current $\omega_T$ satisfying {\rm(\ref{M-A: continuity})} at $T$,

\item[{\rm(2)}] $\omega_T$ is smooth outside the exceptional locus $M_{\exc}$ of $\pi$,

\item[{\rm(3)}] the metric completion of $(M\backslash M_{\exc},\omega_T)$ is a compact length metric space, denoted $(M_T,d_T)$, which is homeomorphic to $X$,

\item[{\rm(4)}] $(M,\omega(t))$ converges in the Cheeger-Gromov sense to $(M_T,d_T)$.
\end{itemize}
\end{theo}

To produce the minimal model program one needs to deform the continuity equation (\ref{M-A: continuity}) through the singular time $T$. When $\pi$ is a divisoral contraction, one can deform the equation on $X$ directly; when $\pi$ is a small contraction, one has to deform the equation on the flip of $M$, so one critical step will be constructing the flips. See Conjectures 4.2 and 4.3 in \cite{LaTi14}. We refer to \cite{SoTi09} for a full description of the minimal model program by K\"ahler-Ricci flow.

\subsection{Algebraic surfaces}

When $M$ is an algebraic surface, the contractions are blow-downs. Hence, exactly as in K\"ahler-Ricci flow \cite{SoWe13}, we have

\begin{theo}\label{theo: surface}
Let $M$ be an algebraic surface and $\omega_0\in 2\pi c_1(L')$ for some line bundle $L'$. Suppose $T<\infty$, then either
\begin{itemize}
\item[{\rm(1)}] $\omega(t)$ collapse at $T$ in the sense that $([\omega_0]+TK_M)^2=0$, or

\item[{\rm(2)}] $\omega(t)$ converges as $t\rightarrow T$ to a blow-down $\pi:M\rightarrow X$, along a finite number of disjoint exceptional curves, in the sense described in Theorem \ref{theo: contraction}.
\end{itemize}
\end{theo}




\begin{thebibliography}{99}


\bibitem{An90} M. Anderson, {\it Convergence and rigidity of manifolds under Ricci curvature bounds}, Invent. Math., 102 (1990), 429-445.


\bibitem{BEGZ} S. Boucksom, P. Eyssidieux, V. Guedj and A. Zeriahi, {\it Monge-Amp\`ere equations in big cohomology classes}, Acta Math., 205 (2010), 199-262.

\bibitem{Ca85} H.D. Cao, {\it Deformation of K\"ahler metrics to K\"ahler-Einstein metrics on compact K\"ahler manifolds}, Invent. Math., 81 (1985), 359-372.

\bibitem{Ch03} J. Cheeger, {\it Integral bounds on curvature, elliptic estimates and rectifiability of singular sets}, Geom. Funct. Anal., 13 (2003), 20-72.

\bibitem{ChCo97} J. Cheeger and T. H. Colding, {\it On the structure of spaces with Ricci curvature bounded below I}, J. Diff. Geom., 46 (1997), 406-480.

\bibitem{ChCo00} J. Cheeger and T. H. Colding, {\it On the structure of spaces with Ricci curvature bounded below II}, J. Diff. Geom., 54 (2000), 13-35.

\bibitem{ChCoTi} J. Cheeger, T. H. Colding and G. Tian, {\it On the singularities of spaces with bounded Ricci curvature}, Geom. Funct. Anal., 12 (2002), 873-914.


\bibitem{ChTi} J. Cheeger and G. Tian, {\it Anti-self-duality of curvature and degeneration of metrics with special holonomy}, Comm. Math. Phys., 255 (2005), 391-417.

\bibitem{Co97} T. H. Colding, {\it Ricci curvature and volume convergece}, Annal. of Math., 145 (1997), 477-501.

\bibitem{CoNa12} T. H. Colding and A. Naber, {\it Sharp H\"older continuity of tangent cones for spaces with a lower Ricci curvature bound and applications}, Anna. of Math., 176 (2012), 1173-1229.

\bibitem{DoSu14} S. Donaldson and S. Sun, {\it Gromov-Hausdorff limits of K\"ahler manifolds and algebraic geometry}, Acta Math. 213 (2014), no. 1, 63¨C106.

\bibitem{EyGuZe09} P. Eyssidieux, V. Guedj and A. Zeriahi, {\it Singular K\"ahler-Einstein metrics}, J. AMS, 22 (2009), 607-639.

\bibitem{FoZh12} F. T. Fong and Z. Zhang, {\it The collapsing rate of the K\"ahler-Ricci flow with regular infinite time singularity}, arXiv:1202.3199

\bibitem{GrHa} P. Griffiths and J. Harris, Principles of algebraic geometry, Wiley-Interscience, New York, 1978.

\bibitem{Gu15} B. Guo, {\it On the K\"ahler Ricci flow on projective manifolds of general type}, arXiv:1501.04239

\bibitem{Ka84} Y. Kawamata, {\it The cone of curves of algebraic varieties}, Ann. of Math., 119 (1984), 603-633.

\bibitem{KoMo} J. Kollar and F. Mori, Birational geometry of algebraic varieties, with the collaboration of C.H. Clemens and A. Corti, Cambridge Tracts in Mathematics, 134. Cambridge University Press, 1998.

\bibitem{Ko98} S. Ko{\l}odziej, {\it The complex Monge-Amp¨¨re equation}, Acta Math. 180 (1998), 69¨C117.

\bibitem{LaTi14} G. La Nave and G. Tian, {\it A continuity method to construct canonical metrics}, arXiv:1410.3157.

\bibitem{CLi} C. Li, K\"ahler-Einstein metrics and K-stability, thesis, 2012.

\bibitem{Li} P. Li, Geometric analysis. Cambridge Studies in Advanced Mathematics, 134. Cambridge University Press, Cambridge, 2012. x+406 pp.

\bibitem{RoZh11} X.C. Rong and Y.G. Zhang, {\it Continuity of extremal transitions and flops for Calabi-Yau manifolds}, J. Diff. Geom., 89 (2011), 233-269.

\bibitem{Ga13} G. Sz\'ekelyhidi, {\it The partial $C^0$-estimate along the continuity method}, arXiv:1310.8471v1

\bibitem{So09} J. Song, {\it Finite time extinction of the K\"ahler-Ricci flow}, arXiv:0905.0939

\bibitem{So13} J. Song, {\it Ricci flow and birational surgery}, arXiv:1304.2607

\bibitem{So14} J. Song, {\it Riemannian geometry of K\"ahler-Einstein currents}, arXiv:1404.0445

\bibitem{So14-2} J. Song, {\it Riemannian geometry of K\"ahler-Einstein currents II: an analytic proof of Kawamata's base point free theorem}, arXiv:1409.8374

\bibitem{SoTi07} J. Song and G. Tian, {\it The K\"ahler-Ricci flow on surfaces of positive Kodaira dimension}, Invent. Math., 170 (2007), 609-653.

\bibitem{SoTi12} J. Song and G. Tian, {\it Canonical measures and K\"ahler-Ricci flow}, J. AMS., 25 (2012), 303-353.

\bibitem{SoTi09} J. Song and G. Tian, {\it The K\"ahler-Ricci flow through singularities}, arXiv:0909.4898

\bibitem{SoWe11} J. Song and B. Weinkove, {\it The K\"ahler-Ricci flow on Hirzebruch surfaces}, J. Reine. Angew. Math., 659 (2011), 141-168.

\bibitem{SoWe13} J. Song and B. Weinkove, {\it Contracting exceptional divisors by the K\"ahler-Ricci flow}, Duke Math. J., 162 (2013), 367-415.

\bibitem{SoWe13-2} J. Song and B. Weinkove, {\it An intorduction to the K\"ahler-Ricci flow}, Lect. Notes in Math., vol. 2086, Springer, 2013, 89-188.

\bibitem{SoWe14} J. Song and B. Weinkove, {\it Contracting exceptional divisors by the K\"ahler-Ricci flow II}, Proc. London Math. Soc., 108 (2014), 1529-1561.

\bibitem{ToWeYa} V. Tosatti, B. Weinkove and X.K. Yang, {\it The K\"ahler-Ricci flow, Ricci-flat metrics and collapsing limits}, arXiv:1408.0161

\bibitem{Ti89} G. Tian, {\it On Calabi's conjecture for complex surfaces with positive first Chern class}. Invent. Math., 101, (1990), 101-172.

\bibitem{Ti08} G. Tian, {\it New Progresses and Results on K\"ahler-Ricci Flow},
G\'eom\'etrie diff\'erentielle, physique math\'ematique, math\'ematiques et soci\'et\'e. II. Ast\'erisque. No. 322 (2008), 71-92.

\bibitem{Ti13} G. Tian, {\it Partial $C^0$-estimates for K\"ahler-Einstein metrics}, Communications in Mathematics and Statistics, 1 (2013), no. 2, 105-113.

\bibitem{Ti12} G. Tian, {\it K-stability and K\"ahler-Einstein metrics}, arXiv:1211.4669

\bibitem{Ti14} G. Tian, {\it An extension of Matsushima's theorem}, preprint.

\bibitem{TiWa12} G. Tian and B. Wang, {\it On the structure of almost Einstein manifolds}, arXiv:1202.2912v1


\bibitem{TiZh06} G. Tian and Z. Zhang, {\it On the K\"ahler-Ricci flow on projective manifolds of general type}, Chin. Ann. Math., 27B (2006), 179-192.

\bibitem{Ya78} S.T. Yau, {\it On the Ricci curvature of a compact K\"ahler manifold and the complex Monge-Amp\`ere equation I}, Commun. Pure Appl. Math., 31 (1978), 339-411.

\bibitem{Ya78-2} S.T. Yau, {\it A general Schwarz lemma for K\"ahler manifolds}, Amer. J. of Math., 100 (1978), 197-203.

\bibitem{Zh06} Z. Zhang, {\it On Degenerate Monge-Ampere Equations over Closed Kahler Manifolds}. Intern. Math. Res. Notices, 2006.

\bibitem{Zh10} Z. Zhang, {\it Scalar curvature behavior for finite-time singularity of K\"ahler-Ricci flow}, Michigan Math. J., 59 (2010), 419-433.

\bibitem{ZhZL10} Z.L. Zhang, {\it Degeneration of shrinking Ricci solitons}. Int. Math. Res. Not. 2010, no. 21, 4137-4158.
\end{thebibliography}
\end{document}